\documentclass{amsart}
\usepackage[margin=1in]{geometry}
\usepackage[utf8]{inputenc}
\usepackage{xcolor}
\usepackage{amsthm}
\usepackage{amssymb}
\usepackage{amsxtra}
\usepackage{xspace} 
\usepackage{comment} 
\usepackage{tikz} 
\usetikzlibrary{cd}
\usepackage{microtype}
\usepackage[colorlinks=true, backref=section]{hyperref}

\numberwithin{equation}{subsection}

\newtheorem{theorem}{Theorem}[section]
\newtheorem*{theorem*}{Theorem}
\newtheorem{lemma}[theorem]{Lemma}
\newtheorem{proposition}[theorem]{Proposition}
\newtheorem{corollary}[theorem]{Corollary}
\newtheorem*{corollary*}{Corollary}
\newtheorem{question}[theorem]{Question}
\newtheorem*{question*}{Question}

\newtheorem*{conjecture*}{Conjecture}

\newtheorem{definition}[theorem]{Definition}
\theoremstyle{remark}
\newtheorem{example}[theorem]{Example}

\newtheorem{remark}[theorem]{Remark}
\theoremstyle{remark}
\newtheorem*{remark*}{Remark}

\DeclareMathOperator{\id}{id}

\newcommand{\op}{\mathrm{op}} 
\newcommand{\ie}{i.e.\xspace}
\newcommand{\eg}{e.g.\xspace}
\newcommand{\loccit}{\textit{loc.\,cit.}\xspace}
\newcommand{\too}{\longrightarrow}

\newcommand{\cA}{\mathcal{A}}

\newcommand{\cC}{\mathcal{C}}

\newcommand{\cE}{\mathcal{E}}
\newcommand{\cF}{\mathcal{F}}
\newcommand{\cG}{\mathcal{G}}

\newcommand{\cK}{\mathcal{K}}
\newcommand{\cL}{\mathcal{L}}

\newcommand{\cO}{\mathcal{O}}

\newcommand{\cU}{\mathcal{U}}

\newcommand{\fA}{\mathfrak{A}}

\newcommand{\fL}{\mathfrak{L}}

\newcommand{\bbG}{\mathbb{G}}

\newcommand{\bbQ}{\mathbb{Q}}

\newcommand{\bbZ}{\mathbb{Z}}

\DeclareMathOperator{\Aut}{Aut}
\newcommand{\Ab}{\mathsf{Ab}} 
\newcommand{\B}{\mathrm{B}} 
\newcommand{\C}{\mathsf{C}} 
\newcommand{\ch}{\mathrm{Ch}} 

\DeclareMathOperator{\Ext}{Ext}
\DeclareMathOperator{\EExt}{\textbf{E}xt}
\DeclareMathOperator{\ger}{\textsc{Gerbes}} 
\newcommand{\Grp}{\mathsf{Grp}} 
\newcommand{\HH}{\mathbf{H}} 

\DeclareMathOperator{\Hom}{Hom}
\DeclareMathOperator{\HOM}{\mathsf{Hom}}
\newcommand{\iso}{\cong}
\renewcommand{\Im}{\operatorname{Im}}
\newcommand{\mmu}{\boldsymbol{\mu}} 
\newcommand{\Ob}{\mathrm{Ob}}
\newcommand{\shC}{\C\sptilde} 
\newcommand{\shAbC}{\shC_\mathrm{ab}} 
\newcommand{\shGrpC}{\shC_\mathrm{grp}} 
\DeclareMathOperator{\Spec}{Spec}
\newcommand{\ts}{\textsuperscript}
\DeclareMathOperator{\Tot}{Tot}
\DeclareMathOperator{\tors}{\textsc{Tors}} 
\newcommand{\T}{\mathsf{T}} 
\newcommand{\X}{\mathsf{X}} 

\title[Cup products and the Heisenberg group]{Cup products, the Heisenberg
  group, and codimension two algebraic cycles}
\author{Ettore Aldrovandi and Niranjan Ramachandran} 

\address{Ettore Aldrovandi, Department of Mathematics, Florida State University, Tallahassee, FL 32306-4510 USA.}
\email{aldrovandi@math.fsu.edu}
\urladdr{http://www.math.fsu.edu/~ealdrov}

\address{Niranjan Ramachandran, Department of Mathematics, University of Maryland, College Park, MD 20742 USA.}
\email{atma@math.umd.edu}
\urladdr{http://www2.math.umd.edu/~atma/}

\date{\today}


\begin{document}
\begin{abstract}We define higher categorical invariants (gerbes) of codimension two algebraic cycles and provide a categorical interpretation of the intersection of divisors on a smooth proper algebraic variety. This generalization of the classical relation between divisors and line bundles furnishes a new perspective on the Bloch-Quillen formula. 
 \end{abstract}
\maketitle

\tableofcontents
\newpage

\begin{flushright}
\itshape Des buissons lumineux fusaient comme des gerbes;\\
Mille insectes, tels des prismes, vibraient dans l'air;\\
Le vent jouait avec l'ombre des lilas clairs,\\
Sur le tissu des eaux et les nappes de l'herbe.\\
Un lion se couchait sous des branches en fleurs;\\
Le daim flexible errait l\`a-bas, pr\`es des panth\`eres;\\
Et les paons d\'eployaient des faisceaux de lueurs\\
Parmi les phlox en feu et les lys de lumi\`ere.\\

\normalfont  --Emile Verhaeren (1855-1916), \\
 \textit{Le paradis} (Les rythmes souverains)
\end{flushright}
 
\section{Introduction}

This aim of this paper is to define higher categorical invariants (gerbes) of codimension two algebraic cycles and provide a categorical interpretation of the intersection of divisors on a smooth proper algebraic variety. This generalization of the classical relation between divisors and line bundles furnishes a new perspective on the classical Bloch-Quillen formula relating Chow groups and algebraic K-theory.

Our work is motivated by the following three basic  questions. 

\begin{enumerate} 
\item Let $A$ and $B$ be abelian sheaves  on a manifold (or algebraic variety) $X$. Given $\alpha \in H^1(X,A)$ and $\beta \in H^1(X,B)$, one has their cup-product $\alpha \cup \beta \in H^2(X, A\otimes B)$. We recall that $H^1$ and $H^2$ classify equivalence classes of torsors and gerbes\footnote{For us, the term "gerbe" signifies a stack in groupoids which is locally non-empty and locally connected (\S \ref{sec:gerbes}).  It is slightly different from the ancient gerbes of  \textit{Acids, alkalies and salts: their manufacture and applications, Volume 2} (1865) by Thomas Richardson and Henry Watts, pp. 567-569: 
\begin{quote}
  \S 4. Gerbes

  This firework is made in various ways, generally throwing up a luminous and sparkling jet of fire, somewhat resembling a 
water-spout: hence its name. Gerbes consist of a straight, cylindrical case, sometime made with wrought iron (if the gerbe is 
of large dimensions). ...  Mr. Darby has invented an entirely novel and beautiful gerbe, called the Italian gerbe..."
\end{quote}}:
\begin{align*}
  H^1(X, A) \quad &\longleftrightarrow \quad \text{Isomorphism classes of $A$-torsors} \\
  H^2(X, A) \quad &\longleftrightarrow \quad \text{Isomorphism classes of $A$-gerbes};
\end{align*}
we may pick torsors $P$ and $Q$ representing $\alpha$ and $\beta$ and ask
\begin{question}\label{Q1}
  Given $P$ and $Q$, is there a natural construction of a gerbe $G_{P,Q}$ which manifests the cohomology class $\alpha \cup\beta =[P] \cup [Q]$?
\end{question}
The above question admits the following algebraic-geometric analogue. 
\item Let $X$ be a smooth proper variety over a field $F$. Let $Z^i(X)$ be the abelian group of algebraic cycles of codimension $i$ on $X$ and let $CH^i(X)$ be the Chow group of algebraic cycles of codimension $i$ modulo rational equivalence.  The  isomorphism
  \begin{equation*}
    CH^1(X) \overset{\sim}{\too} H^1(X, \cO^*) 
  \end{equation*}
  connects (Weil) divisors and invertible sheaves (or $\bbG_m$-torsors). While divisors form a group,  $\bbG_m$-torsors on $X$ form a Picard category $\tors_X(\bbG_m)$ with the monoidal structure provided by the Baer sum of torsors. Any divisor $D$ determines a $\bbG_m$-torsor $\cO_D$;  the torsor 
 $\cO_{D +D'}$ is isomorphic to the Baer sum of $\cO_D$ and $\cO_{D'}$. In other words, one has an additive map \cite[II, Proposition 6.13]{Hartshorne}
 \begin{equation}\label{zee-one}
 Z^1(X) \to \tors_X(\bbG_m) \qquad D \mapsto \cO_D.
 \end{equation}
  
\begin{question}\label{Q2}
What is a natural generalization of \eqref{zee-one} to higher codimension cycles?
\end{question}
Since $\tors_X(\bbG_m)$ is a Picard category, one could expect the putative additive maps  on $Z^i(X)$ to land in Picard categories or their generalizations.  
\begin{question}\label{Q3}  Is there a categorification of the intersection pairing
\begin{equation}\label{int-pairing}
CH^1(X) \times CH^1(X) \to CH^2(X)?
\end{equation}
\end{question}
More generally, one can ask for a categorical interpretation of the entire Chow ring of $X$. 
 \end{enumerate}

\subsection*{Main results} 

 Our first result is an affirmative answer to Question \ref{Q1}; the key observation is that a certain Heisenberg group \emph{animates} the cup-product. 

\begin{theorem}
  \label{primo}
  Let $A, B$ be abelian sheaves on a topological space or scheme $X$.
  \begin{enumerate}
  \item 
 There is a canonical functorial Heisenberg\footnote{The usual  Heisenberg group, a central extension of $A \times B$ by $\mathbb C^*,$ arises from a biadditive map $A \times B \to \mathbb C^*$.}  sheaf $H_{A,B}$ on $X$ which sits in an exact sequence
 \begin{equation*}
   0 \to A\otimes B \to H_{A,B} \to A \times B \to 0;
 \end{equation*}
 the  sheaf $H_{A,B}$ (of non-abelian groups) is a central extension of $A \times B$ by $A\otimes B$.

\item The associated boundary map
  \begin{equation*}
    \partial: H^1(X, A)\times H^1(X,B) = H^1(X, A \times B) \to H^2(X, A\otimes B)
  \end{equation*}
  sends the class $(\gamma, \delta)$ to the cup-product $\gamma \cup \delta$.
 
\item Given torsors $P$ and $Q$ for $A$ and $B$, view $P\times Q$ as a $A\times B$-torsor on $X$. Let $\cG_{P,Q}$ be the gerbe of local liftings (see \S \ref{lifting}) of $P\times Q$ to a $H_{A,B}$-torsor; its band is $A\otimes B$ and its class in $H^2(X, A\otimes B)$ is $[P]\cup[Q]$. 

\item The gerbe $\cG_{P,Q}$ is covariant functorial in $A$ and $B$ and contravariant functorial in $X$. 
\item The gerbe $\cG_{P,Q}$ is trivial (equivalent to the stack of $A\otimes B$-torsors) if either $P$ or $Q$ is trivial.
\end{enumerate}
\end{theorem}
We prove this theorem over a general site $\C$. We also provide a natural interpretation of the (class of the) Heisenberg sheaf in terms of maps of 
Eilenberg-Mac~Lane objects in \S \ref{sec:cup}; it is astonishing that the explicit cocycle \eqref{eq:3} for the Heisenberg group (when $X =$ a point) turns out to coincide with the map on the level of Eilenberg-Mac~Lane objects over a general site $\C$; cf. \ref{prop:cup}.   

Here is another rephrasing of Theorem \ref{primo}: For abelian sheaves $A$ and $B$ on a site $\C$, there is a natural bimonoidal functor 
\begin{equation}\label{bimon} \tors _{\C}(A)\times \tors _{\C}(B) \too \ger _{\C} (A\otimes B) \qquad (P,Q) \mapsto \cG_{P,Q}\end{equation}
where $\tors_{\C}(A)$, $\tors_{\C}(B)$ are the Picard categories of $A$ and $B$-torsors on $\C$ and $\ger_{\C} (A\otimes B)$ is the Picard $2$-category of $A\otimes B$-gerbes on $\C$. Thus, Theorem \ref{primo} constitutes a categorification of the cup-product map 
 \begin{equation}
 \cup: H^1(A) \times H^1(B) \to H^2(A\otimes B).
 \end{equation}

Let us turn to Questions \ref{Q2} and \ref{Q3}. Suppose that $D$ and $D'$ are divisors on $X$ which intersect in the codimension-two cycle $D.D'$. Applying Theorem \ref{primo} to $\cO_D$ and $\cO_{D'}$ with $A = B =\bbG_m$, one has a $\bbG_m \otimes \bbG_m$-gerbe $\cG_{D, D'}$ on $X$.   
We now invoke the isomorphisms (the second is the fundamental Bloch-Quillen isomorphism)
\begin{equation*}
  \bbG_m \overset{\sim}{\too} \cK_1, \qquad CH^i(X) \underset{\eqref{BQ}}{\overset{\sim}{\too}} H^i(X, \cK_i)
\end{equation*}
where $\cK_i$ is the Zariski sheaf associated with the presheaf $U \mapsto K_i(U)$. 
  
Pushforward of $\cG_{D, D'}$ along $\cK_1 \times \cK_1 \to \cK_2$ gives a $\cK_2$-gerbe still denoted $\cG_{D, D'}$; we call this the Heisenberg gerbe attached to the codimension-two cycle $D.D'$. This raises the possibility of relating $\cK_2$-gerbes and codimension-two cycles on $X$, implicit in \eqref{BQ}. 

\begin{theorem}\label{secondo} (i) Any codimension-two cycle $\alpha \in Z^2(X)$ determines a $\cK_2$-gerbe $\cC_{\alpha}$ on $X$.

(ii) the class of $\cC_{\alpha}$ in $H^2(X, \cK_2)$ corresponds to $\alpha\in CH^2(X)$ under the Bloch-Quillen map \eqref{BQ}.

(iii)  the gerbe $\cC_{\alpha + \alpha'}$ is equivalent to the Baer sum of $\cC_{\alpha}$ and $\cC_{\alpha'}$. 

(iv) $\cC_{\alpha}$ and $\cC_{\alpha'}$ are equivalent as $\cK_2$-gerbes if and only if $\alpha =\alpha'$ in $CH^2(X)$.
\end{theorem}
The Gersten gerbe $\cC_{\alpha}$ of $\alpha$ admits a geometric description, closely analogous to that of the $\bbG_m$-torsor $\cO_D$ of a divisor $D$; see Remark \ref{empty}. The Gersten sequence \eqref{gersten} is key to the construction of $\cC_{\alpha}$. One has an additive map 
\begin{equation}\label{zee-two} Z^2(X) \to \ger_X(\cK_2) \qquad \alpha \mapsto \cC_{\alpha}.\end{equation}

When $\alpha = D.D'$ is the intersection of two divisors, there are two $\cK_2$-gerbes attached to it: the Heisenberg gerbe $\cG_{D,D'}$ and the Gersten gerbe $\cC_{\alpha}$; these are abstractly equivalent as their classes in $H^2(X, \cK_2)$ correspond to $\alpha$. More is possible. 

\begin{theorem}\label{terzo}
If $\alpha \in Z^2(X)$ is the intersection $D.D'$ of divisors $D, D' \in Z^1(X)$, then there is a natural equivalence $\Theta: \cC_{\alpha} \to \cG_{D, D'}$ between the Gersten and Heisenberg $\cK_2$-gerbes attached to $\alpha = D.D'$.
\end{theorem}
Thus, Theorems \ref{primo}, \ref{secondo}, \ref{terzo} together provide the following commutative diagram thereby answering Question \ref{Q3}:

\begin{equation*}
  \begin{tikzcd}
    Z^1(X) \times Z^1(X)  \arrow[r, dotted, "{\textrm{no~map}}"] \arrow{d}{\eqref{zee-one}} & Z^2(X) \arrow{d}{\eqref{zee-two}}\\
    \tors_X(\bbG_m) \times \tors_X(\bbG_m) \arrow{r}{\eqref{bimon}} \arrow{d} & \ger_X(\cK_2) \arrow{d}\\
    CH^1(X) \times CH^1(X) \arrow{r}{\eqref{int-pairing}} & CH^2(X).
  \end{tikzcd}
\end{equation*}

We begin with a review of the basic notions and tools (lifting gerbe, four-term
complexes) in \S \ref{sec:prel} and then present the construction and properties
of the Heisenberg group in \S \ref{sec:heis} before proving Theorem \ref{primo}.
After a quick discussion of various examples in \S \ref{sec:Examples}, we turn
to codimension-two algebraic cycles in \S \ref{sec:codim2} and construct the
Gersten gerbe $\cC_{\alpha}$ and prove Theorems \ref{secondo}, \ref{terzo} using
the tools in \S \ref{sec:prel}.

\subsection*{Dictionary for codimension two cycles}
The above results indicate the viability of
viewing $\cK_2$-gerbes as natural invariants of codimension-two cycles on
$X$. Additional evidence is given by the following points: \footnote{Let $\eta \colon \textrm{Spec}~F_X\to X$ be the generic point of $X$ and write $K_i^{\eta}$ for the sheaf $\eta_* K_i(F_X)$; one has the map $\cK_i \to K_i^{\eta}$.}

\begin{itemize}
\item $\cK_2$-gerbes are present (albeit implicitly) in the Bloch-Quillen formula \eqref{BQ} for $i=2$. 

\item The  Picard category $\mathfrak P =\tors_X(\bbG_m)$ of $\bbG_m$-torsors on $X$ satisfies
  \begin{equation*}
    \pi_1(\mathfrak P) = H^0(X, \cO^*) = CH^1(X, 1), \qquad \pi_0(\mathfrak P) = H^1(X, \cO^*)= CH^1(X).
  \end{equation*}
  Similarly, the Picard $2$-category $\mathfrak C = \ger_X (\cK_2)$ of $\cK_2$-gerbes is closely related to Bloch's higher Chow complex \cite{Bloch2} in codimension two:
  \begin{equation*}
    \pi_2(\mathfrak C) = H^0(X, \cK_2)= CH^2(X,2) ,\quad \pi_1(\mathfrak C) = H^1(X, \cK_2) = CH^2(X, 1), \quad \pi_0(\mathfrak C) = H^2(X, \cK_2) \overset{\eqref{BQ}}{=} CH^2(X).
  \end{equation*}

\item The additive map arising from Theorem \ref{secondo}
  \begin{equation*}
    Z^2(X) \to \ger_X(\cK_2), \qquad \alpha \mapsto \cC_{\alpha}
  \end{equation*}
  gives the Bloch-Quillen isomorphism \eqref{BQ} on the level of $\pi_0$. It provides an answer to Question \ref{Q2} for codimension two cycles.

\item The Gersten gerbe $\cC_{\alpha}$ admits a simple algebro-geometric description (Remark \ref{trivial-eta}): 
Any $\alpha$ determines a $K_2^{\eta}/{\cK_2}$-torsor; then $\cC_{\alpha}$ is the gerbe of liftings of this torsor  to a $K_2^{\eta}$-torsor on $X$.

\item The gerbe $\cC_{\alpha}$ is canonically trivial outside of the support of $\alpha$ (Remark \ref{trivial-eta}).
\item Pushing the Gersten gerbe $\cC_{\alpha}$ along the map $\cK_2 \to \Omega^2$ produces an $\Omega^2$-gerbe which manifests the (de Rham) cycle class of $\alpha$ in $H^2(X, \Omega^2)$.  
\end{itemize}
The map \eqref{zee-one} is a part of the marvellous dictionary  \cite[II,
\S6]{Hartshorne} arising from the divisor sequence~\eqref{gersten1}:
\begin{equation*}
  \text{Divisors}\quad \longleftrightarrow \quad \text{Cartier~divisors}\quad
  \longleftrightarrow \quad \text{$\cK_1$-torsors} \quad \longleftrightarrow \quad\text{Line~bundles}\quad \longleftrightarrow \quad \text{Invertible~sheaves}.
\end{equation*}
More generally, from the Gersten sequence~\eqref{gersten} we obtain the following:
\begin{gather*}
  Z^1(X) \overset{\iso}{\too} H^0(X, K_1^{\eta}/\cK_1) \twoheadrightarrow H^1(X,\cK_1) \iso
  CH^1(X) \\
  Z^2(X) \twoheadrightarrow H^1(X, K_2^{\eta}/\cK_2) \overset{\iso}{\too} H^2(X,\cK_2) \iso CH^2(X) .
\end{gather*}

Inspired by this and by ref.\ \cite[Definition 3.2]{Bloch}, we call $K_2^{\eta}/\cK_2$-torsors as \emph{codimension-two Cartier cycles} on $X$. Thus the analog for codimension two cycles of the above dictionary reads
\begin{equation*}
  \text{Codimension two cycles}\quad \longleftrightarrow \quad \text{Cartier~cycles}\quad \longleftrightarrow \quad\text{$\cK_2$-gerbes}.
\end{equation*}

Since the Gersten sequence \eqref{gersten} exists for all $\cK_i$, it  is possible to generalize Theorem \ref{secondo} to higher codimensions thereby answering Question \ref{Q2}; however, this involves \textit{higher gerbes}. Any cycle of codimension $i>2$ determines a higher gerbe \cite{Breen94} with band $\cK_i$ (see \S \ref{end} for an example); this provides a new perspective on the Bloch-Quillen formula \eqref{BQ}. The higher dimensional analogues of  \eqref{bimon}, \eqref{int-pairing}, and Theorem \ref{secondo} will be pursued elsewhere. 

Other than the classical Hartshorne-Serre correspondence between certain
codimension-two cycles and certain rank two vector bundles, we are not aware of
any generalizations of this dictionary to higher codimension. In particular, our
idea of attaching \emph{a higher-categorical invariant to a higher codimension
  cycle} seems new in the literature. We expect that Picard $n$-categories play
a role in the functorial Riemann-Roch program of Deligne \cite{Deligne0}.

Our results are related to and inspired by the beautiful work of S.~Bloch \cite{Bloch}, L.~Breen \cite{Breen99}, J.-L.~Brylinski \cite{Brylinski}, A.~N.~Parshin \cite{Parshin}, B.~Poonen - E.~Rains \cite{PoonenRains}, and D.~Ramakrishnan \cite{Ramakrishnan} (see \S \ref{sec:Examples}).  Brylinski's hope\footnote{``In principle such ideas will lead to a geometric description of all regulator maps, once the categorical aspects have been cleared up. Hopefully this would lead to a better understanding of algebraic K-theory itself.''} \cite[Introduction]{Brylinski} for a higher-categorical geometrical interpretation of the  regulator maps from algebraic K-theory to  Deligne cohomology was a major catalyst.  In a forthcoming paper, we will  investigate the  relations between the Gersten gerbe and Deligne cohomology. 

\subsection*{Acknowledgements.} The second author's research was supported by the  2015-2016 ``Research and Scholarship Award''  from the
Graduate School, University of Maryland. We would like to thank J.~Rosenberg, J.~Schafer, and H.~Tamvakis for comments and suggestions.

\subsection*{Notations and conventions} Let $\C$ be a site. We write $\shC$ for
the topos of sheaves over $\C$, $\shAbC$ the abelian group objects of $\shC$,
namely the abelian sheaves on $\C$, and by $\shGrpC$ the sheaves of groups on
$\C$. Our notation for cohomology is as follows. For an abelian object $A$ of a
topos $\T$, $H^i(A)$ denotes the cohomology of the terminal object $e\in\T$ with
coefficients in $A$, namely $i^\mathrm{th}$ derived functor of
$\Hom_\T(e,A)$. This is the same as $\Ext^i_{\T_\mathrm{ab}}(\bbZ,A)$.  More
generally, $H^i(X,A)$ denotes the cohomology of $A$ in the topos $\T/X$. We use
$\HH$ for hypercohomology.

\section{Preliminaries}\label{sec:prel}

\subsection{Abelian Gerbes \cite{Giraud,Deligne,Breen94}}
\label{sec:gerbes} 
A gerbe $\cG$ over a site $\C$ is \textit{a stack in groupoids which is locally non-empty and locally connected}.

$\cG$ is locally nonempty if for every object $U$ of $\C$ there is a cover, say a local epimorphism, $V\to U$ such that the category $\cG(V)$ is nonempty; it is locally connected if given objects $x,y\in \cG(U)$ as above, then, locally on $U$, the sheaf $\Hom(x,y)$ defined above has sections. For each object $x$ over $U$ we can introduce the automorphism sheaf $\Aut_\cG(x)$, and by local connectedness all these automorphism sheaves are (non canonically) isomorphic.

In the sequel we will only work with \emph{abelian} gerbes, where there is a coherent identification between the automorphism sheaves $\Aut_\cG(x)$, for any choice of an object $x$ of $\cG$, and a fixed sheaf of groups $G$. In this case $G$ is necessarily abelian\footnote{The automorphisms in $\Aut(G)$ completely decouple, hence play no role.}, and the class of $\cG$ determines an element in $H^2(G)$, \cite[\S 2]{Breen94} (and also \cite{jardine}), where $H^i(G)=\Ext^i_{\shAbC}(\bbZ,G)$ denotes the standard cohomology with coefficients in the abelian sheaf $G$ in the topos $\shC$ of sheaves over $\C$.

Let us briefly recall how the class of $\cG$ is obtained using a \v{C}ech type argument. Assume for simplicity that the site $\C$ has pullbacks. Let $\cU=\{U_i\}$ be a cover of an object $X$ of $\C$. Let $x_i$ be a choice of an object of $\cG(U_i)$. For simplicity, let us assume that we can find morphisms $\alpha_{ij}\colon x_j\rvert_{U_{ij}}\to x_i\rvert_{U_{ij}}$. The class of $\cG$ will be represented by the 2-cocycle $\{c_{ijk}\}$ of $\cU$ with values in $G$ obtained in the standard way as the deviation for $\{\alpha_{ij}\}$ from satisfying the cocycle condition:
\begin{equation*}
  \alpha_{ij}\circ \alpha_{jk} = c_{ijk}\circ \alpha_{ik}.
\end{equation*}
In the above identity---which defines it---$c_{ijk}\in \Aut(x_i\rvert_{U_{ijk}})\iso G\rvert_{U_{ijk}}$. It is obvious that  $\{c_{ijk}\}$ is a cocycle.

Returning to stacks for a moment, a stack $\cG$ determines an object $\pi_0(\cG)$, defined as the sheaf associated to the presheaf of connected components of $\cG$, where the latter is the presheaf that to each object $U$ of $\C$ assigns the set of isomorphism classes of objects of $\cG(U)$. By definition, if $\cG$ is a gerbe, then $\pi_0(\cG)=*$. In general, writing just $\pi_0$ in place of $\pi_0(\cG)$, by base changing to $\pi_0$, namely considering the site $\C/\pi_0$, every stack $\cG$ is (tautologically) a gerbe over $\pi_0$ \cite{L-MB}.

\begin{example}\hfill
  \begin{enumerate}
  \item The trivial gerbe with band $G$ is the stack $\tors(G)$ of $G$-torsors.  Moreover, for any gerbe $\cG$, the choice of an object $x$ in $\cG(U)$ determines an equivalence of gerbes $\cG\rvert_U\iso \tors (G\rvert_U)$, over $\C/U$, where $G=\Aut_\cG(x)$. There is an equivalence $\tors(G)\iso \B_G$, the topos of (left) $G$-objects of $\shC$ (\cite{Giraud}).

  \item Any line bundle $L$ over an algebraic variety $X$ over $\bbQ$ determines a gerbe $\cG_n$ with band $\mmu_n$ (the sheaf of $n$\ts{th} roots of unity) for any $n >1$ as follows: Over any open set $U$, consider the category of pairs $(\cL, \alpha)$ where $\cL$ is a line bundle on $U$ and $\alpha: \cL ^{\otimes n} \xrightarrow{\sim} L$ is an isomorphism of line bundles over $U$. These assemble to the gerbe $G_n$ of $n$\ts{th} roots of $L$. This is an example of a lifting gerbe \S \ref{lifting}.
  \end{enumerate}
\end{example}

\begin{remark*} One also has the following interpretation, which shows that, in a fairly precise sense, a gerbe is the categorical analog of a torsor. Let $\cG$ be a gerbe
  over $\C$, let $\{U_i\}$ be a cover of $U\in \Ob(\C)$, and let $\{x_i\}$ be a
  collection of objects $x_i\in \cG(U_i)$. The $G$-torsors
  $E_{ij}=\Hom(x_j,x_i)$ are part of a ``torsor cocycle'' $\gamma_{ijk}\colon
  E_{ij} \otimes E_{jk}\to E_{ik}$, locally given by $c_{ijk}$, above, and
  subject to the obvious identity. Let $\tors(G)$ be the stack of $G$-torsors
  over $X$. Since $G$ is assumed abelian, $\tors(G)$ has a group-like
  composition law given by the standard Baer sum. The fact that $\cG$ itself is locally equivalent to $\tors(G)$, plus the datum of the torsor cocycle $\{E_{ij}\}$, show that $\cG$ is equivalent to a $\tors(G)$-torsor.
\end{remark*}

The primary examples of abelian gerbes occurring in this paper are the gerbe of
local lifts associated to a central extension and four-term complexes, described
in the next two sections.

\subsection{The gerbe of lifts associated with a central extension}
\label{lifting}
(See \cite{Giraud,Breen94,Brylinskib}.)
A central extension
\begin{equation}
  \label{eq:7}
  0 \too A \overset{\imath}{\too} E\overset{p}{\too} G \too 0
\end{equation}
of sheaves of groups determines a homotopy-exact sequence
\begin{equation*}
  \tors(A) \too \tors(E)\too \tors(G),
\end{equation*}
which is an extension of topoi with characteristic class $c \in H^2(\B_G,A)$. (Recall that $A$ is abelian and that $\tors(G)$ is equivalent to $\B_G$.)
If $\X$ is any topos over $\tors(G)\iso \B_G$, the gerbe of lifts is the gerbe with band $A$
\begin{equation*}
  \cE = \HOM_{\B_G}(\X,\B_E),
\end{equation*}
where $\HOM$ denotes the cartesian morphisms. The class $c(\cE)\in H^2(\X,A)$ is the pullback of $c$ along the map $\X\to \B_G$. By the universal property of $\B_G$, the morphism $\X\to \B_G$ corresponds to a $G$-torsor $P$ of $\X$, hence the $A$-gerbe $\cE$ is the gerbe whose objects are (locally) pairs of the form $(Q,\lambda)$, where $Q$ is an $E$-torsor and $\lambda\colon Q\to P$ an equivariant map. It is easy to see that an automorphism of an object $(Q,\lambda)$ can be identified with an element of $A$, so that $A$ is indeed the band of $\cE$.

Let us take $\X=\shC$, and let $P$ be a $G$-torsor. With the same assumptions as the end of \S~\ref{sec:gerbes}, let $X$ be an object of $\C$ with a cover $\lbrace U_i\rbrace$. In this case, the class of $\cE$ is computed by choosing $E\rvert_{U_i}$-torsors $Q_i$ and equivariant maps $\lambda_i\colon Q_i\to P\rvert_{U_i}$. Up to refining the cover, let $\alpha_{ij}\colon Q_j\to Q_i$ be an $E$-torsor isomorphism such that $\lambda_i\circ \alpha_{ij}=\lambda_j$. With these choices the class of $\cE$ is given by the cocycle $\alpha_{ij}\circ \alpha_{jk} \circ \alpha_{ik}^{-1}$.

\begin{remark}
  \label{rem:2}
  The above argument gives the well known boundary map \cite[Proposition 4.3.4]{Giraud}
  \begin{equation*}
    \partial^1\colon H^1(G) \too H^2(A)
  \end{equation*}
  (where we have omitted $\X$ from the notation). Dropping down one degree we get [ibid., Proposition 3.3.1]
  \begin{equation*}
    \partial^0\colon H^0(G) \too H^1(A).
  \end{equation*}
  In fact these are just the boundary maps determined by the above short exact sequence when all objects are abelian.  The latter can be specialized even further: if $g\colon *\to G$, then by pullback the fiber $E_g$ is an $A$-torsor \cite{GrothendieckSGA7}.
\end{remark}

\subsection{Four-term complexes}
\label{four-term}

Let $\shAbC$ be the category of abelian sheaves over the site $\C$. Below we shall be interested in four-term exact sequences of the form:
\begin{equation}
  \label{eq:6}
  0 \too A \overset{\imath}{\too} L_1\overset{\partial}{\too}
  L_0 \overset{p}{\too} B\too 0.
\end{equation}
Let $\ch_+(\shAbC)$ be the category of positively graded homological complexes of abelian sheaves. The above sequence can be thought of as a (non-exact) sequence 
\begin{equation*}
  0 \too A[1] \too [L_1\too L_0] \too B \too 0
\end{equation*}
of morphisms of $\ch_+(\shAbC)$. This sequence is short-exact in the sense of Picard categories, namely as a short exact sequence of Picard stacks
\begin{equation*}
  0 \too \tors(A) \too \cL \overset{p}{\too} B\too 0,
\end{equation*}
where $\cL$ is the strictly commutative Picard stack associated to the complex $L_1\to L_0$ and the abelian object $B$ is considered as a discrete stack in the obvious way. We have isomorphisms $A\iso \pi_1(\cL)$ and $B \iso \pi_0(\cL)$, where the former is the automorphism sheaf of the object $0\in \cL$ and the latter the sheaf of connected components (see \cite{Breen94,Breen99,DeligneDG}). It is also well known that the projection $p\colon \cL\to B$ makes $\cL$ a \emph{gerbe} over $B$. In this case the band of $\cL$ over $B$ is $A_B$, thereby determining a class in $H^2(B,A)$.\footnote{This is \emph{part} of the invariant classifying the four-term sequence, see the remarks in \cite[\S 6]{Breen99}.}

Rather than considering $\cL$ itself as a gerbe over $B$, we shall be interested in its fibers above generalized points $\beta\colon *\to B$. Let us put $\cA=\tors(A)$.
By a categorification of the arguments in \cite{GrothendieckSGA7},
the fiber $\cL_\beta$ above $\beta$ is an $\cA$-torsor, hence an $A$-gerbe, by the observation at the end of \S~\ref{sec:gerbes} (see also the equivalence described in \cite{Breen90}). $\cL_\beta$ is canonically equivalent to $\cA$ whenever $\beta=0$. Writing
\begin{equation*}
  \Hom_{\shC}(*,B) \iso \Hom_{\shAbC}(\bbZ,B) = H^0(B),
\end{equation*}
we have the homomorphism
\begin{equation}\label{doppio}
 \partial^2:  H^0(B) \too H^2(A),
\end{equation}
which sends $\beta$ to the class of $\cL_\beta$ in $H^2(A)$. The sum of $\beta$ and $\beta'$ is sent to the Baer sum of $\cL_\beta+\cL_{\beta'}$, and the characteristic class is additive.  In the following Lemma we show this map is the same as the one described in \cite[Théorème 3.4.2]{Giraud}.
\begin{lemma}
  \label{comparisons}
  \hfill
  \begin{enumerate}
  \item The map $\partial^2$ in \eqref{doppio} is the canonical cohomological map (iterated boundary map) \cite[Théorème 3.4.2]{Giraud}
    \begin{equation*}
      d^2: H^0(B) \too H^1(C) \too H^2(A)
    \end{equation*}
    ($C$ is defined below) arising from the four-term complex \eqref{eq:6}.
  \item The image of $\beta$ under $d^2$ is the class of  the gerbe $\cL_{\beta}$. 
  \end{enumerate}
\end{lemma}
\begin{proof}
  We keep the same notation as above. Let us split~\eqref{eq:6} as
  \begin{equation*}
    \begin{tikzcd}
      &&& 0 \arrow[d] & \\
      0 \arrow[r] & A \arrow[r, "\imath"] & L_1 \arrow[r, "\pi"] \arrow[dr, "\partial"'] & C \arrow[r] \arrow[d, "\jmath"] & 0 \\
      &          &            & L_0 \arrow[d, "p"] &  \\
      &          &            & B \arrow[d] & \\
      &          &            & 0 &
    \end{tikzcd}
  \end{equation*}
  with $C= \Im \partial$. By Grothendieck's theory of extensions~\cite{GrothendieckSGA7}, with $\beta \colon *\to B$, the fiber $(L_0)_\beta$ is a $C$-torsor (see the end of Remark~\ref{rem:2}). According to section~\ref{lifting}, we have a morphism $\tors (L_1)\to \tors (C)$, and the object $(L_0)_\beta$ of $\tors(C)$ gives rise to the gerbe of lifts  $\cE_\beta\equiv \cE_{L_0,\beta}$, which is an $A$-gerbe. Now, consider the map assigning to $\beta\in H^0(B)$ the class of $\cE_\beta\in H^2(A)$. By construction, this map factors through $H^1(C)$ by sending $\beta$ to the class of the torsor $(L_0)_\beta$. We then lift that to the class of the gerbe of lifts in $H^2(A)$. All stages are compatible with the abelian group structures. This is the homomorphism described in \cite[Théorème 3.4.2]{Giraud}.  

  It is straightforward that this is just the classical lift of $\beta$ through the four-term sequence~\eqref{eq:6}. Indeed, this is again easily seen in terms of a \v Cech cover $\{U_i\}$ of $*$. Lifts $x_i$ of $\beta\rvert_{U_i}$ are sections of the $C$-torsor $(L_0)_\beta$, therefore determining a standard $C$-valued 1-cocycle $\{c_{ij}\}$. From section~\ref{lifting} we then obtain an $A$-valued 2-cocycle $\{a_{ijk}\}$ arising from the choice of local $L_1$-torsors $X_i$ such that $X_i\to (L_0)_\beta \rvert_{U_i}$ is ($L_1\to C$)-equivariant. Note that in the case at hand, $\pi\colon L_1\to C$ being an epimorphism, the lifting of the torsor $(L_0)_\beta$ is done by choosing local trivializations, \ie the $x_i$ above, and then choosing $X_i=L_1\rvert_{U_i}$.

  The same argument shows that the class of $\cL_\beta$, introduced earlier, is the same as that of $\cE_\beta$. This follows from the following well known facts: objects of $\cL_\beta$ are locally lifts of $\beta$ to $L_0$; morphisms between them are given by elements of $L_1$ acting through $\partial$. As a result, automorphisms are sections of $A$ and clearly the class so obtained coincides with that of $\cE_\beta$.  Therefore $\cE_\beta$ and $\cL_\beta$ are equivalent and the homomorphism of \cite[Théorème 3.4.2]{Giraud} is equal to~\eqref{doppio}, as required.
\end{proof}
From the proof of the above lemma, we obtain the following  two descriptions of the $A$-gerbe $\cL_{\beta}$. 
 \begin{corollary}\label{diversi} (i) For any four-term complex  \eqref{eq:6} and any generalized point $\beta$ of $B$, the fiber  $\cL_{\beta}$ is a gerbe. Explicitly, it is the stack associated with the prestack which attaches to $U$ the groupoid $\cL_{\beta}(U)$ whose objects are elements $g \in L_0(U)$ with $p(g) = \beta$ and morphisms between $g$ and $g'$ given by elements $h$ of $L_1(U)$ satisfying $\partial(h) = g-g'$.
 
 (ii)  The $A$-gerbe $\cL_{\beta}$ is the lifting gerbe of the $C$-torsor $(L_0)_{\beta}$ to a $L_1$-torsor. \qed
\end{corollary}
We will use both descriptions in \S \ref{sec:codim2} especially in the comparison of the Gersten and the Heisenberg gerbe of a codimension two cycle, in the case that it  is an intersection of divisors. 

A slightly different point of view is the following. Recast the sequence~\eqref{eq:6} as a quasi-isomorphism
\begin{equation*}
  A[2] \overset{\iso}{\too} \bigl[ L_1\too L_0\too B \bigr]
\end{equation*}
of three-term complexes of $\ch_+(\shAbC)$, where now $A$ has been shifted two places to the left. Also, relabel the right hand side as $L'_2\to L'_1\to L'_0$ (where again we employ homological degrees) for convenience. By \cite{Tatar}, the above morphism of complexes of $\ch_+(\shAbC)$, placed in degrees $[-2,0]$, gives an equivalence between the corresponding associated strictly commutative Picard 2-stacks
\begin{equation*}
  \fA \overset{\iso}{\too} \fL
\end{equation*}
over $\C$. Here $\fL = [L'_2\to L'_1\to L'_0]\sptilde$ and $\fA = [A\to 0\to 0]\sptilde \iso \tors (\cA)\iso \ger (A)$. This time we have $\pi_0(\fL) = \pi_1(\fL) = 0$, and $\pi_2(\fL)\iso A$, as it follows directly from the quasi-isomorphism above. Thus $\fL$ is 2-connected, namely any two objects are locally (\ie after base change) connected by an arrow; similarly, any two arrows with the same source and target are---again, locally---connected by a 2-arrow.

Locally, any object of $\fL$ is a section $\beta\in B=L'_0$. By the preceding argument, the Picard stack $\cL_\beta = \Aut_\fL(\beta)$ is an $A$-gerbe, and the assignment $\beta \mapsto \cL_\beta$ realizes (a quasi-inverse of) the equivalence between $\fA$ and $\fL$. It is easy to see that $\cL_\beta$ is the same as the fiber over $\beta$ introduced before.

In particular, for the Gersten resolution~\eqref{gersten}, \eqref{eq:9}, for $\cK_2$,
we get the equivalence of Picard 2-stacks
\begin{equation}
  \label{eq:8}
  \ger (\cK_2) \iso \bigl[G_2^X \bigr]\sptilde.
\end{equation}

\section{The Heisenberg group}\label{sec:heis}

The purpose of this section is to describe a functor $H\colon \Ab\times \Ab\to \Grp$, where $\Ab$ is the category of abelian groups and $\Grp$ that of groups.
If $\C$ is a site, the method immediately generalizes to the categories of abelian groups and of groups in $\shC$, the topos of sheaves on $\C$. For any pair $A, B$ of abelian sheaves on $\C$, there is a canonical Heisenberg sheaf $H_{A,B}$ (of non-commutative groups on $\C$), a central extension of $A \times B$ by $A\otimes B$.

The definition of $H$ is based on a generalization of the Heisenberg group construction due to Brylinski \cite[\S 5]{Brylinski}. A  pullback along the diagonal map $A \to A\otimes A$ gives the extension constructed by Poonen and Rains \cite{PoonenRains}.

\subsection{The Heisenberg group}
\label{sec:elem}

Let $A$ and $B$ be abelian groups. Consider the (central) extension
\begin{equation}
  \label{eq:1}
  0 \to A\otimes B \to H_{A,B} \to A\times B \to 0
\end{equation}
where the group $H_{A,B}$ is defined by the group law:
\begin{equation}
  \label{eq:2}
  (a,b,t)\, (a',b',t') = (aa',bb',t + t' + a\otimes b').
\end{equation}
Here $a,a'$ are elements of $A$, $b,b'$ of $B$, and $t,t'$ of $A\otimes B$. The nonabelian group $H_{A,B}$ is evidently a functor of the pair $(A,B)$, namely a pair of homomorphisms $(f\colon A\to A', g\colon B\to B')$ induces a homomorphism $H_{f,g}\colon H_{A,B}\to H_{A',B'}$. The special case $A=B=\mmu_n$ occurs in Brylinski's treatment of the regulator map to \'etale cohomology \cite{Brylinski}.

The map
\begin{equation}
  \label{eq:3}
  f \colon (A\times B)\times (A\times B) \too A\otimes B,
  \qquad
  f(a,b,a',b') = a\otimes b',
\end{equation}
is a cocycle representing the class of the extension~\eqref{eq:1} in $H^2(A\times B, A\otimes B)$ (group cohomology).  Its alternation
\begin{equation*}
  \varphi_f\colon \wedge_\bbZ^2 (A\times B) \too A\otimes B,
  \qquad
  \varphi_f((a,b),(a',b')) = a\otimes b' - a'\otimes b,
\end{equation*}
coincides with the standard commutator map and represents the value of the projection of the class of $f$ under the third map in the universal coefficient sequence
\begin{equation*}
  0 \too \Ext^1(A\times B, A\otimes B) \too
  H^2(A\times B, A\otimes B) \too
  \Hom(\wedge_\bbZ^2 (A\times B), A\otimes B).
\end{equation*}
As for the commutator map, it is equal to $[s,s]\colon \wedge_\bbZ^2 (A\times B)\to A\otimes B$, where $s\colon A\times B\to H_{A,B}$ is a set-theoretic lift, but the map actually is independent of the choice of $s$. (For details see, \eg the introduction to \cite{Breen99}.)
\begin{remark}
  The properties of the class of the extension $H_{A,B}$, in particular that it is a cup-product of the fundamental classes of $A$ and $B$, as we can already evince from~\eqref{eq:3}, are best expressed in terms of Eilenberg-Mac~Lane spaces. We will do this below working in the topos of sheaves over a site.
\end{remark}
\subsection{Extension to sheaves}
\label{sec:sh} 

The construction of the Heisenberg group carries over to the sheaf context. Let $\C$ be a site, and $\shC$ the topos of sheaves over $\C$. Denote by $\shAbC$ the abelian group objects of $\shC$, namely the abelian sheaves on $\C$, and by $\shGrpC$ the sheaves of groups on $\C$.

For all pairs of objects $A,B$ of $\shAbC$, it is clear that the above construction of $H_{A,B}$ carries over to a functor
\begin{equation*}
  H\colon \shAbC\times \shAbC \too \shGrpC.
\end{equation*}
In particular, since $H_{A,B}$ is already a sheaf of sets (isomorphic to $A \times B \times (A\otimes B)$), the only question is whether the group law varies nicely, but this is clear from its functoriality. Note further that by definition of $H_{A,B}$ the resulting epimorphism $H_{A,B}\to A\times B$ has a global section $s\colon A\times B\to H_{A,B}$ as objects of $\shC$, namely $s=(\id_A,\id_B,0)$, which we can use this to repeat the calculations of \S~\ref{sec:elem}.

In more detail, from \S~\ref{lifting}, the class of the central extenson~\eqref{eq:1} is to be found in
$H^2(\B_{A\times B},A\otimes B)$ ($A\otimes B$ is a trivial $A\times B$-module). This replaces the group cohomology of \S~\ref{sec:elem} with its appropriate topos equivalent. By pulling back to the ambient topos, say $\X=\shC$, this is the class of the gerbe of lifts from $B_{A\times B}$ to $B_H$. We are ready to give a proof of Theorem~\ref{primo}. This proof is computational.

\begin{proof}[Proof of Theorem~\ref{primo}]
  Let us go back to the cocycle calculations at the end of \S~\ref{lifting}, where $X$ is an object of $\C$ equipped with a cover $\cU=\lbrace U_i\rbrace$. An $A\times B$-torsor $(P,Q)$ over $X$ would be represented by a \v{C}ech cocycle $(a_{ij},b_{ij})$ relative to $\cU$. The cocycle is determined by the choice of isomorphisms $(P,Q)\rvert_{U_i}\iso (A\times B)\rvert_{U_i}$. Now, define $R_i = H_{A,B}\rvert_{U_i}$ with the trivial $H_{A,B}$-torsor structure, and let $\lambda_i\colon R_i\to (P,Q)\rvert_{U_i}$ equal the epimorphism in~\eqref{eq:1}. Carrying out the calculation described at the end of~\ref{lifting} with these data gives $\alpha_{ij}\circ \alpha_{jk}\circ \alpha_{ik}^{-1}=a_{ij}\otimes b_{jk}$, which is the cup-product in \v Cech cohomology of the classes corresponding to the $A$-torsor $P$ and the $B$-torsor $Q$. In other words, the gerbe of lifts corresponding to the central extension determined by the Heisenberg group incarnates the cup product map
  \begin{equation*}
    H^1(X,A) \times H^1(X,B) \overset{\cup}{\too} H^2(X,A\otimes B).
  \end{equation*}
  For the choice $\alpha_{ij} = (a_{ij}, b_{ij}, 0)$, one has the following explicit calculation in the Heisenberg group
  \begin{align*}
    \alpha_{ij}\circ \alpha_{jk}\circ \alpha_{ik}^{-1}
    & = (a_{ij}, b_{ij}, 0)(a_{jk}, b_{jk}, 0)(a_{ik}, b_{ik}, 0)^{-1}\\
    & = (a_{ik}, b_{ik}, a_{ij}\otimes b_{jk}) (a^{-1}_{ik}, b^{-1}_{ik}, a_{ik}\otimes b_{ik})\\
    & = (1,1, a_{ij} \otimes b_{jk} + a_{ik}\otimes b_{ik} - a_{ik} \otimes b_{ik}) \\
    & = (1,1, a_{ij} \otimes b_{jk});
  \end{align*}
  We used that the inverse of $(a, b, t)$ in the Heisenberg group is $(a^{-1}, b^{-1}, -t + a\otimes b)$:
  \begin{equation*}
    (a, b, t) (a^{-1}, b^{-1}, -t + a\otimes b)
    = (1, 1, a\otimes b^{-1} + t -t + a\otimes b) = (1,1,0).
  \end{equation*}
  It is well known \cite[Chapter 1, \S 1.3, Equation (1-18), p.~29]{Brylinskib} that
the \v Cech cup-product of $a =\{a_{ij}\}$ and $b = \{b_{ij}\}$ is given by the two-cocycle
\begin{equation*}
  \{a\cup b\}_{ijk} = \{a_{ij}\otimes b_{jk}\}.
\end{equation*}
 
  This proves the first three points of the statement, whereas the fourth is built-in from the very construction. The fifth follows from the fact that the class of the gerbe of lifts is bilinear: this is evident from the expression computed above.
\end{proof}

As hinted above, the cup product has a more intrinsic explanation in terms of maps between Eilenberg-Mac~Lane objects in the topos. Passing to Eilenberg-Mac Lane objects in particular ``explains'' why the cup-product realizes the cup-product pairing. First, we state
\begin{theorem}
  \label{thm:cup}
  The class of the extension~\eqref{eq:1} in $\shC$ corresponds to (the homotopy class of) the cup product map
  \begin{equation*}
    K(A\times B,1) \iso K(A,1)\times K(B,1) \too K(A\otimes B,2)
  \end{equation*}
  between the identity maps of $K(A,1)$ and $K(B,1)$; its expression is given by~\eqref{eq:3}.
\end{theorem}
\begin{proof}
  Observe the epimorphism $H_{A,B}\to A\times B$ has global set-theoretic sections. The statement follows from Propositions~\ref{simpltop} and~\ref{prop:cup} below.
\end{proof}
The two main points, which we now proceed to illustrate, are that Eilenberg-Mac~Lane objects represent cohomology (and hypercohomology, once we take into account simplicial objects) in a topos, and that the cohomology of a group object in a topos (such as $A\times B$ in $\shC$) with trivial coefficients can be traded for the hypercohomology of a simplicial model of it. In this way we calculate the class of the extension as a map, and such map is identified with the cup product. We assemble the necessary results to flesh out the proof of Theorem~\ref{thm:cup} in the next two sections.

\subsection{Simplicial computations}
\label{sec:simplicial}
The class of the central extension~\eqref{eq:7} can be computed simplicially. (For the following recollections, see \cite[VI.5, VI.6, VI.8]{IllusieII} and \cite[\S 2]{Breen}.)

Let $\T$ be a topos, $G$ a group-object of $\T$ (for us it will be $\T=\shC$) and $BG=K(G,1)$ the standard classifying simplicial object with $B_nG=G^n$~\cite{DeligneH3}. Let $A$ be a trivial $G$-module. We will need the following well known fact.\footnote{Unfortunately we could not find a specific entry point in the literature to reference, therefore we assemble here the necessary prerequisites. See also \cite[\S\S 2,3]{chinburg} for a detailed treatment in the representable case.}
\begin{proposition}
  \label{simpltop}
  \begin{math}H^i(\B_G,A) \iso \HH^i(BG,A).
  \end{math}
\end{proposition}
\begin{proof}
The object on the right is the hypercohomology as a simplicial object of $\T$. Let $X$ be a simplicial object in a topos $\T$. One defines
\begin{equation*}
  \HH^i(X,A) = \EExt^i(\bbZ[X]\sptilde, A)
\end{equation*}
where $M\sptilde$, for any simplicial abelian object $M$ of $\T$, denotes the corresponding chain complex defined by $M_n\sptilde=M_n$, and by taking the alternate sum of the face maps. $\bbZ X_n$ denotes the abelian object of $\T$ generated by $X_n$. Of interest to us is the spectral sequence \cite[Example (2.10) and below]{Breen}:
\begin{equation*}
  E_1^{p,q} = H^q(X_p,A) \Longrightarrow \HH^\bullet(X, A).
\end{equation*}

Let $X$ be any simplicial object of $\T$. The levelwise topoi $\T/X_n$, $n=0,1,\dots$, form a simplicial topos $\X=\T/X$ or equivalently a topos fibered over $\Delta^\op$, where $\Delta$ is the simplicial category. The topos $\B\X$ of $\X$-objects essentially consists of descent-like data, that is, objects $L$ of $\X_0$ equipped with an arrow $a\colon d_1^*L\to d_0^*L$ the cocycle condition $d_0^*a\, d_2^*a=d_1^*a$ and $s_0^*a = \id$ (the latter is automatic if $a$ is an isomorphism).
By~\cite[VI.8.1.3]{IllusieII}, in the case where $X=BG$, $\B\X$ is nothing but $\B_G$, the topos of $G$-objects of $\T$. One also forms the topos $\Tot(\X)$, whose objects are collections $F_n\in \X_n$ such that for each $\alpha\colon [m]\to [n]$ in $\Delta^\op$ there is a morphism $F_\alpha\colon \alpha^*F_m\to F_n$, where $\alpha^*$ is the inverse image corresponding to the morphism $\alpha\colon \X_n\to \X_m$. There is a functor $\mathit{ner}\colon \B\X\to \Tot(\X)$ sending $(L,a)$ to the object of $\Tot(\X)$ which at level $n$ equals $(d_0\cdots d_0)^*L$ ($a$ enters through the resulting face maps), see \loccit for the actual expressions. The functor $\mathit{ner}$ is the inverse image functor for a morphism $\Tot(\X)\to \B\X$, and, $\X$ satisfying the conditions of being a ``good pseudo-category'' (\cite[VI 8.2]{IllusieII}) we have an isomorphism
\begin{equation*}
  R\Gamma(\B\X,L) \overset{\iso}{\too} R\Gamma (\Tot(\X),\mathit{ner}(L))
\end{equation*}
and, in turn, a spectral sequence
\begin{equation*}
  E_1^{p,q} = H^q(X_p,\mathit{ner}_p(L)) \Longrightarrow H^\bullet(\B\X, L),
\end{equation*}
\cite[VI, Corollaire 8.4.2.2]{IllusieII}. On the left hand side we recognize the spectral sequence for the cohomology of a simplicial object in a topos \cite[\S 2.10]{Breen}.  

Applying the foregoing to $X=BG$, and $L$ a left $G$-object of $\T$, we obtain \cite[VI.8.4.4.5]{IllusieII}
\begin{equation*}
  E_1^{p,q} = H^q(G^p,L) \Longrightarrow H^\bullet(\B_G, L).
\end{equation*}
(We set $Y=e$, the terminal object of $\T$, in the formulas from \loccit) 

Thus if $L=A$, the trivial $G$-module arising from a central extension of $G$ by $A$, by comparing the spectral sequences we can trade $H^2(\B_G, A)$ for the hypercohomology $\HH^2(K(G,1),A)$. 
\end{proof}

\subsection{The cup product}
\label{sec:cup}
The class of the extension extension~\eqref{eq:1} corresponds to the homotopy class of a map $K(A\times B,1)\to K(A\otimes B,2)$.  We interpret it in terms of cup products of Eilenberg-Mac~Lane objects. 

Recall that  for an object $M$ of $\shAbC$ we have $K(M,i)=K(M[i])$, where $M[i]$ denotes $M$ placed in homological degree $i$, and $K \colon \ch_+(\shAbC)\to s\shAbC$ is the Dold-Kan functor from nonnegative chain complexes of $\shAbC$ to simplicial abelian sheaves. Explicitly:
  \begin{equation*}
    K(M,i)_n =
    \begin{cases}
      0 & 0\leq n < i, \\
      \bigoplus_{s \colon [n] \twoheadrightarrow [i]} M & n\geq i.
    \end{cases}
  \end{equation*}
  In particular, $K(M,i)_i=M$. $K$ is a quasi-inverse to the normalized complex functor $N\colon s\shAbC\to \ch_+(\shAbC)$.

  If $X$ is a simplicial object $X$ of $\shC$, we have
  \begin{equation}
    \label{eq:5}
  \HH^i(X,M) \iso [X,K(M,i)],
\end{equation}
where the right-hand side denotes the hom-set in the homotopy category \cite{Illusie,Breen}. In particular, there is a  fundamental class $\imath_M^n\in \HH^n(K(M,n),M)$, corresponding to the identity map.

Returning to the objects $A$ and $B$ of $\shAbC$, also recall the morphism \cite[Chapter II, Equation (2.22), p.~64]{Breen}
\begin{equation}
  \label{eq:4}
  \delta_{i,j}\colon
  K(A,i)\times K(B,j) \too K(A\otimes B,i+j).
\end{equation}
It is the composition of two maps. The first is:
\begin{equation*}
  K(A,i)\times K(B,j)\too d((K(A,i)\boxtimes K(B,j)) = (K(A,i)\otimes K(B,j))),
\end{equation*}
where $\boxtimes$  denotes the external tensor product of simplicial objects of $\shAbC$ and $d$ the diagonal; the second is the map in $s\shAbC$ corresponding to the Alexander-Whitney map under the Dold-Kan correspondence. We have:
\begin{proposition}
  \label{prop:cup}
  The class of the extension~\eqref{eq:1} is equal to $\imath_A^1\otimes \imath_B^1 = \delta_{1,1}(\imath_A^1\times \imath_B^1)$.
\end{proposition}
\begin{proof}
  Observe that any simplicial morphism $f\colon X\to K(M,i)$ is determined by $f_i$, the rest, for $n>i$, being determined by the simplicial identities. Therefore we need to compute:
  \begin{equation*}
    K(A\times B,1)_2 \iso K(A,1)_2\times K(B,1)_2 \too K(A\otimes B,2)_2,
  \end{equation*}
  namely
  \begin{equation*}
    (A\times B)\times (A\times B) \too (A\times A)\times (B\times B) \too A\otimes B.
  \end{equation*}
  From the expression of the Alexander-Whitney map, in \eg\ \cite{Illusie}, the image of the second map in $\ch_+(\shAbC)$ is the sum of $d_0^vd_0^v$, $d_1^hd_1^h$, and $d_2^hd_0^v$. Only the third one is nonzero, giving $((a,b),(a',b'))\to a\otimes b'$, which equals $f$ in the construction of the extension~\eqref{eq:1}.  Using~\eqref{eq:5} we obtain the conclusion.
\end{proof}
The morphism~\eqref{eq:4} represents the standard cup product in cohomology. By Proposition ~\ref{prop:cup}, for an object $X$ of $s\shC$, the cup product
\begin{equation*}
  \HH^1(X,A)\times \HH^1(X,B) \too \HH^2(X,A\otimes B)
\end{equation*}
factors through $X\to K(A,1)\times K(B,1)$ and the extension~\eqref{eq:1}.
\begin{remark*}
  Proposition~\ref{prop:cup} and the above map provide a more conceptual proof of Theorem~\ref{primo}.
\end{remark*}

\section{Examples and connections to prior results}
\label{sec:Examples}
In this section, we collect some examples and briefly indicate the connections with earlier results \cite{Bloch, Brylinski, Parshin, PoonenRains, Ramakrishnan}.

\subsection{Self-cup products of Poonen-Rains}   In \cite{PoonenRains}, Poonen and Rains construct, for any abelian group $A$, a central extension of the form
  \begin{equation*}
    0 \to A\otimes A\to U\!A\to A \to 0,
  \end{equation*}
  providing a functor $U\colon \Ab\to \Grp$.  The group law in $U\!A$ is obtained from~\eqref{eq:2} by setting $a=a'$ and $b=b'$. Hence the above extension can be obtained from~\eqref{eq:1} by pulling back along the diagonal homomorphism $\Delta_A\colon A\to A\times A$. Similarly, both the cocycle and its alternation for the extension constructed in \loccit\ are obtained from ours by pullback along $\Delta_A$, for $A\in \Ab$. Similar remarks apply over an abelian sheaf $A$ on any site $\C$. They use  $U\!A$ to describe the self-cup product $\alpha \cup \alpha$ of any element $\alpha \in H^1(A)$.

\subsection{Brylinski's work on regulators and \'etale analogues} In \cite{Brylinski}, Brylinski has proved Theorem \ref{primo} in the case $A =B =\mmu_n$, the \'etale sheaf $\mmu_n$ of $n$\ts{th} roots of unity on a scheme $X$ over $\Spec\bbZ[\frac{1}{n}]$ using the Heisenberg group $H_{\mmu_n, \mmu_n}$ (in our notation). He used it to provide a geometric interpretation of the regulator map
\begin{equation*}
  c_{1,2}: H^1(X, \cK_2) \too H^3(X, \mmu_n^{\otimes 2}), \qquad  (\text{$n$ odd}),
\end{equation*}
a special case of C.~Soul\'e's regulator. If $X$ is a smooth projective variety over $\mathbb C$ (viewed as an complex analytic space) and $f,g$ are invertible functions on $X$, P.~Deligne (and Bloch) \cite{Deligne0} constructed a holomorphic line bundle $(f,g)$ on $X$ and Bloch showed that this gives a regulator map from $K_2(X)$ to the group of isomorphism classes of holomorphic line bundles with connection, later interpreted by D.~Ramakrishnan \cite{Ramakrishnan} in terms of the three-dimensional Heisenberg group. 

Write $[f]_n, [g]_n \in H^1(X, \mmu_n)$ for the images of $f,g$ under the boundary map $H^0(X, \cO_{X^{an}}) \to H^1(X, \mmu_n)$ of the analytic Kummer sequence
\begin{equation*}
  1\too \mmu_n \too \cO_{X^{an}}^* \xrightarrow{u\mapsto u^n} \cO_{X^{an}}^* \too 1.
\end{equation*}
The gerbe $G_{[f]_n,[g]_n}$ from Theorem \ref{primo} is compatible with Bloch-Deligne line bundle $(f,g)$, in a sense made precise in \cite[Proposition 5.1 and after]{Brylinski}.

\subsection{Finite flat group schemes} Let $X$ be any variety over a perfect field $F$ of characteristic $p>0$. For any commutative finite flat group scheme $N$ killed by $p^n$, consider the cup product pairing
\begin{equation*}
  H^1(X, N) \times H^1(X, N^D) \to H^2(C, \mmu_{p^n})
\end{equation*}
of flat cohomology groups where $N^D$ is the Cartier dual of $N$. Theorem \ref{primo} provides a $\mmu_{p^n}$-gerbe on $X$ given a $N$-torsor and a $N^D$-torsor.  When $N$ is the kernel of $p^n$ on an abelian scheme $A$ so that $N^D$ is the kernel of $p^n$ on the dual abelian scheme $A^D$ of $A$, the cup-product pairing is related to the N\'eron-Tate pairing \cite[p.~19]{Milne67}.

 \subsection{The gerbe associated with a pair of divisors}\label{pairs} Let $X$ be a smooth variety over a field $F$. Let $D$ and $D'$ be divisors on $X$. Consider the non-abelian sheaf $H$ on $X$ 
obtained by pushing the Heisenberg group $H_{\cK_1, \cK_1}$ along the multiplication map $m: \cK_1\otimes \cK_1 \to \cK_2$. 
So $H$ is a central extension of $\cK_1 \times K_1$ by $\cK_2$ which we write
\begin{equation*}
  0 \too \cK_2 \too H \overset{\pi}{\too} \cK_1 \times \cK_1 \too 0.
\end{equation*}

 Let $L = L_{D, D'}$ denote the $\cK_1 \times \cK_1$-torsor defined by the pair $D, D'$. Applying Theorem \ref{primo} gives a $\cK_2$-gerbe on $X$ as follows. Since $H$ is a central extension (so $\cK_1 \times \cK_1$ 
 acts trivially on $\cK_2$), the category of local liftings of $L$ to a $\cK_2$-torsor provide (\S \ref{lifting}, \cite[IV, 4.2.2]{Giraud}) a canonical $\cK_2$-gerbe $\cG_{D, D'}$. 
 \begin{definition}
   The Heisenberg gerbe $\cG_{D, D'}$ with band $\cK_2$ is the following: For each open set $U$, the category $\cG_{D, D'}(U)$ has objects pairs $(P, \rho)$ where $P$ is a $H$-torsor on $U$ and
   \begin{equation*}
     \rho\colon P \times_{\pi} (\cK_1 \times \cK_1) \overset{\sim}{\too} L
   \end{equation*}
   is an isomorphism of $\cK_1 \times \cK_1$-torsors; a morphism from $(P, \rho)$ to 
$(P', \rho')$ is a map $f: P \to P'$ of $H$-torsors satisfying $\rho = \rho'\circ f$. It is clear that the set of morphisms from $(P, \rho)$ to $(P', \rho')$ is a $\cK_2$-torsor.
\end{definition} 

\begin{example} Assume $X$ is a curve (smooth proper) and put $Y =X \times X$.

(i) Assume $F=\mathbb F_q$ is a finite field. Let $D$ be the graph on $Y$ of the Frobenius morphism $\pi: X \to X$ and $D'$ be the diagonal, the image of $X$ under the map $\Delta: X \to X \times X$. Theorem \ref{primo} attaches a $\cK_2$-gerbe on $Y$ to the zero-cycle $D.D'$, the intersection of the divisors $D$ and $D'$. Since the zero cycle $D.D'$ is the pushforward $\Delta_{*} \beta$ of  $\beta= \displaystyle\sum_{x\in X(\mathbb F_q)} x$ on $X$, we obtain that the set of rational points on $X$ determines a $\cK_2$-gerbe on $X \times X$. 

(ii) Note that the diagonal $\Delta_Y$ (a codimension-two cycle on $Y \times Y$) can be written as an intersection of divisors $V$ and $V'$ on $Y\times Y = X \times X\times X \times X$ where $V$ (resp. $V'$) are the set of points of the latter of the form  $\{(a,b,a,c)\}$ (resp. $\{(a,b,d,b)\}$).  Theorem \ref{primo} says that $\Delta_Y$ determines a $\cK_2$-gerbe on $Y\times Y$. 
\end{example}

\subsection{Adjunction formula} Let $X$ be a smooth proper variety and $D$ be a smooth divisor of $X$.  The classical adjunction formula states: 

 \textit{The restriction of the line bundle $L_D^{-1}$ to $D$ is the conormal bundle $N_D$ (a line bundle on $D$).} 
 
 Given a pair of smooth divisors $D, D'$ with $E = D \cap D'$ smooth of pure codimension two, write $\iota: E \hookrightarrow X$ for the inclusion. There is a map $\pi: \iota^*\cK_2 \to \cK_2^E$, where $\cK_2^E$ indicates the usual K-theory sheaf $\cK_2$ on $E$. 
 An analogue of the adjunction formula for $E$ would be a description of the $\cK_2^E$-gerbe $\pi_*  \iota^*\cG_{D,D'}$ obtained from the $\cK_2$-gerbe $\cG_{D, D'}$ on $X$.   
 \begin{proposition}
   Let $D$ and $D'$ be smooth divisors of $X$ with $E = D \cap D'$ smooth of pure codimension two. 
   Consider the line bundles $V = (N_D)|_E$ and $V'= (N_{D'})|_E$ on $E$. Then, $\pi_*  \iota^*\cG_{D, D'}$ is equivalent to the $\cK^E_2$-gerbe $\cG_{V, V'}$.
\end{proposition}
\begin{proof} 
  Since the restriction map $H^*(X, \cK_i) \to H^*(E, \cK^E_i)$ respects cup-product, this follows from the classical adjunction formula for $D$ and $D'$. 
\end{proof}

\subsection{Parshin's adelic groups} Let $S$ be a smooth proper surface over a field $F$. 
For any choice of a curve $C$ in $S$ and a point $P$ on $C$, Parshin
\cite[(18)]{Parshin} has introduced a discrete Heisenberg group
\begin{equation*}
  0 \to \bbZ \to \tilde{\Gamma}_{P,C} \to \Gamma_{P,C} \to 0,
\end{equation*}
where $\Gamma_{P,C}$ is isomorphic (non-canonically) to $\bbZ \oplus \bbZ$;  he has shown \cite[end of
\S 3]{Parshin} how a suitable product of these groups leads to an adelic
description of $CH^2(S)$ and the intersection pairing \eqref{int-pairing}.  His
constructions are closely related to an adelic resolution of the sheaf
$H_{\cK_1, \cK_1}$ on $S$. 

\section{Algebraic cycles of codimension two}\label{sec:codim2}
Throughout this section,  $X$ is a smooth proper variety over a field $F$. Let $\eta \colon \textrm{Spec}~F_X\to X$ be the generic point of $X$ and write $K_i^{\eta}$ for the sheaf $j_* K_i(F_X)$.

In this section, we construct the Gersten gerbe $\cC_{\alpha}$ for any codimension two cycle $\alpha$ on $X$, provide various equivalent 
descriptions of $\cC_{\alpha}$ and use them to prove Theorems \ref{k2gerbe}, \ref{comparison}.  As a consequence, we obtain Theorems \ref{secondo} and \ref{terzo} of the introduction. 

 \subsection{Bloch-Quillen formula} \label{sec:bloch-quillen} Recall  the (flasque) Gersten resolution\footnote{This resolution exists for any smooth variety over $F$.}  \cite[\S7]{Quillen} \cite[p. 276]{handbook} \cite{Gerstenicm} of the Zariski sheaf $\cK_i$ associated with the presheaf $U \mapsto K_i(U)$:

\begin{equation}\label{gersten}
0 \too \cK_i \too \bigoplus_{x \in X^{(0)}} j_* K_i(x) \too \bigoplus_{x \in X^{(1)}} j_* K_{i-1}(x) \too \cdots \bigoplus_{x \in X^{(i-1)}} j_* K_1(x) \xrightarrow{\delta_{i-1}} \bigoplus_{x \in X^{(i)}}{\oplus} ~j_* K_0(x);
\end{equation}
here, any point $x\in X^{(m)}$ corresponds to a subvariety of codimension $m$
and the map $j$ is the canonical inclusion $x \hookrightarrow X$. So $\cK_i$ is
quasi-isomorphic to the complex
\begin{equation}
  \label{eq:9}
  G_i^X =  \bigl[ K_i^{\eta} \too \bigoplus_{x \in X^{(1)}}
  j_* K_{i-1}(x) \too \cdots \bigoplus_{x \in X^{(i-1)}} j_* K_1(x)
  \xrightarrow{\delta_{i-1}} \bigoplus_{x \in X^{(i)}} j_* K_0(x) \bigr] .
\end{equation}

By (\ref{gersten}), there is a functorial isomorphism \cite[\S7, Theorem 5.19]{Quillen} \cite[Corollary 72, p.~276]{handbook} 
\begin{equation}\label{BQ}
  \bigoplus_i CH^i(X) \xrightarrow{\sim} \bigoplus_i H^i(X, \cK_i); \qquad \text{\makebox[0pt][l]{(Bloch-Quillen formula)}}
\end{equation}
this is an isomorphism of graded rings: D.~Grayson has proved that the intersection product on $CH(X) = \oplus_i CH^i(X)$ corresponds to the cup-product in cohomology \cite[Theorem 77, p.278]{handbook}. Thus, algebraic cycles of codimension $n$ give $n$-cocycles of the sheaf $\cK_n$ on $X$ and that two such cocycles are cohomologous exactly when the algebraic cycles are rationally equivalent.  

The final two maps in \eqref{gersten} arise essentially from the valuation and the tame symbol map \cite[pp.351-2]{Bloch}. Let $R$ be a discrete valuation ring, with fraction field $L$; let ${\rm ord}: L^{\times} \to \bbZ$ be the valuation and let $l$ be the residue field. The boundary maps from the localization sequence for $\Spec R$ are known explicitly: the map $L^{\times} = K_1(L) \to K_0(l) = \bbZ$ is the map ${\rm ord}$ and the map $K_2(L) \to K_1(l) = l^{\times}$ is the tame symbol. This applies for any normal subvariety $V$ (corresponds to a $y \in X^{(i)}$) and a divisor $x$ of $V$ (corresponding to a $x \in X^{(i+1)}$). 

\subsection{Divisors}\label{divisors} We recall certain well known results about divisors and line bundles for comparison with  the results below for the $\cK_2$-gerbes attached to codimension two cycles. 

If $A$ is a sheaf of abelian groups on $X$, then ${\rm Ext}_X^1(\mathbb Z, A) = H^1(X, A)$ classifies $A$-torsors on $X$. Given an extension $E$
\begin{equation*}
  0 \too A \too E \overset{\pi}{\too} \bbZ \too 0
\end{equation*}
of abelian sheaves on $X$, the corresponding $A$-torsor is simply $\pi^{-1}(1)$ (a sheaf of sets). When $X$ is a point, then $\pi^{-1}(1)$ is a coset of $\pi^{-1}(0) = A$, i.e., a $A$-torsor. The classical correspondence \cite{Hartshorne} between Weil divisors (codimension-one algebraic cycles) $D$ on $X$, Cartier divisors,  line bundles $\cL_D$, and  torsors $\cO_D$ over $\cO_X^* =\bbG_m =\cK_1$ comes from the Gersten sequence (\ref{gersten}) for $\cK_1$ (see also \cite[2.2]{Gerstenicm}): 
\begin{equation}\label{gersten1}
 0 \too \cO_X^* \too F_X^{\times} \overset{d}{\too} \bigoplus_{x \in X^{(1)}} j_*\bbZ \to 0,
 \end{equation}
 where $F_X$ is the constant sheaf of rational functions on $X$ and the sum is over all irreducible effective divisors on $X$, using that $K_0(L) \cong \bbZ$ and $K_1(L) = L^{\times}$ for any field $L$. As a Weil divisor $D = \Sigma_{x\in X^1}~ n_x x$ is a formal combination with integer coefficients of subvarieties of codimension one of $X$, it determines a map of sheaves
 \begin{equation*}
   \psi\colon \bbZ \too \bigoplus_{x \in X^{(1)}} j_*\bbZ;
 \end{equation*}
 $\psi(1)$ is the section with components $n_x$. The $\cO_X^*$-torsor $\cO_D$ attached to $D$ is given as the subset
\begin{equation}\label{divisore}
{d}^{-1}(\psi(1)) \subset F_X^{\times}.
\end{equation}

 A \v Cech description of $\cO_D$ relative to an Zariski open cover $\{U_i\}$ of $X$ is as follows. Pick a rational function $f_i$ on $U_i$ with pole of order $n_x$ along $x$ for all $x \in U^{(1)}_i$ (so $x$ is a irreducible subvariety of codimension one of $U_i$); we view $f_i \in F_X^{\times}$. On $U_i \cap U_j$, one has $f_i = g_{ij}f_j$ for unique $g_{ij} \in \cO_X^*(U_i \cap U_j)$; the collection $\{g_{ij}\}$ is a \v Cech one-cocycle with values in $\cO_X^*$ representing $\cO_D$. 

For any $D$, $\cL_D$ is trivial on the complement of the support of $D$.

\begin{remark}\label{picard}
  For each open $U$ of $X$, one has the  Picard category $\tors_U(\cO^*)$ of $\cO^*$-torsors on $U$. These combine to the Picard stack $\tors(\cO^*)$ of $\cO^*$-torsors 
on $X$.  The Gersten sequence incarnates this Picard stack \cite[1.10]{DeligneB}. \qed
\end{remark}

\subsection{The Gersten gerbe of a codimension two cycle} 
 We next show that every cycle $\alpha$ of codimension two on $X$ determines a gerbe $\cC_{\alpha}$ with band $\cK_2$.  The Gersten complex \eqref{gersten} enables us to give a geometric description of $\cC_{\alpha}$; see Remark \ref{trivial-eta} below. 
 
  The cycle $\alpha$ provides a natural map
  \begin{equation}\label{gersten4}
    \begin{tikzcd}
      &&&& {\bbZ} \arrow[d, red, "\phi" blue] \\
      0 \arrow[r] &\cK_2 \arrow[r, "\mu"] &  K_2^{\eta} \arrow[r, "\nu"] & {\bigoplus\limits_{x \in X^{(1)}} j_* K_1(x)} \arrow[r, "{\delta}"]  &\bigoplus\limits_{x \in X^{(2)}} j_* K_0(x) \arrow[r] & 0\\
    \end{tikzcd}
  \end{equation}
and an exact sequence (by pullback)
\begin{equation}\label{TEE}
  0 \too \cK_2 \too K_2^{\eta} \overset{\nu}{\too} T \overset{\delta}{\too} \bbZ \too 0.
\end{equation} 
This two-extension of $\bbZ$ by $\cK_2$ gives a class in $\Ext^2(\bbZ,\cK_2) = H^2(X, \cK_2)$. 
Writing $\alpha = \sum_x n_x [x]$ as a sum over  $x\in X^{(2)}$ (irreducible codimension two subvarieties), then the $x$-component of $\phi (1)$ corresponds to $n_x$ under the canonical isomorphism $K_0(x) \cong \bbZ$.  The maps $\delta$ and $\nu$ are essentially given by the valuation (or ${\rm ord}$) and tame symbol maps; see \S \ref{sec:bloch-quillen}. 

\begin{definition}
  The gerbe $\cC_{\alpha}$ (associated with the cycle $\alpha$) is obtained by applying the results of \S \ref{four-term} to  \eqref{gersten4}, \eqref{TEE}; thus it is an example of the gerbe $\cL_{\beta}$ of \S \ref{four-term}, where $\beta=\phi$ and $\cL$ is the Picard stack associated to the complex $[K_2^{\eta}\to \bigoplus\limits_{x \in X^{(1)}} j_* K_1(x)]$.
\end{definition}

\begin{remark} Corollary \ref{diversi} provides two descriptions of $\cL_{\beta}$. It should be emphasized that both descriptions are useful. One of them, which we make explicit below, is crucial for the comparison with the Heisenberg gerbe (Theorem \ref{comparison}); the other succinct description  is given in Remark \ref{trivial-eta}.

  \begin{enumerate}
  \item For any open set $U$ of $X$, the category $\cC_{\alpha}(U)$ has objects $u \in \bigoplus\limits_{x \in X^{(1)}} j_* K_1(x)$ with $\delta{u} = \phi(1)$ and morphisms from $u$ to $u'$ are elements $a \in K_2^{\eta}$ with $\nu(a) = u'-u$. 

  \item Any Hom-set ${\rm Hom}_{\cC_{\alpha}}(u, u')$ is a $K_2(U)$-torsor.

  \item The category $\cC_{\alpha}(U)$ can be described geometrically in terms of the ${\rm ord}$ and tame maps. For instance, let $X$ be a surface. Write the zero-cycle $\alpha$ as a finite sum $ \sum_{i\in I} n_i x_i$ of points $x_i$ of $X$. We assume $n_i \neq 0$ and write $V$ for the complement of the support of $\alpha$. Any non-zero rational function $f$ on a curve $C$ defines an object of $\cC_{\alpha}(U)$ if $f$ is invertible on $C \cap U\cap V$ and satisfies ${\rm ord}_{x_i} f = n_i$ for each  $x_i \in U$ (assuming, for simplicity, that $x_i$ is a smooth point of $C$). A general object of $\cC_{\alpha}(U)$ is a finite collection $u = \{C_j, f_j\}$ of curves $C_j$ and non-zero rational functions $f_j$ on $C_j$ such that $f_j$ is invertible on $C_j \cap U \cap V$ and $\sum {\rm ord}_{x_i} f_j = n_i$ (an index $j$ occurs in the sum if $x_i \in C_j$) for each  $x_i \in U$. A morphism from  $u$ to $u'$ is an element $a \in K_2^{\eta}$ whose tame symbol is $u' - u$. \qed 
  \end{enumerate}
\end{remark}

\begin{theorem}\label{k2gerbe} (i) $\cC_{\alpha}$ is a gerbe on $X$ with band $\cK_2$. 

(ii) Under \eqref{BQ}, the class of $\cC_{\alpha} \in H^2(X, \cK_2)$ corresponds to $\alpha \in CH^2(X)$.

(iii) $\cC_{\alpha}$ is equivalent to $\cC_{\alpha'}$ (as gerbes) if and only if the cycles $\alpha$ and $\alpha'$ are rationally equivalent.\end{theorem}

\begin{proof}  (i) The Gersten sequence \eqref{gersten4} is an example of a four-term complex, discussed in \S \ref{four-term}. As the stack $\cC_{\alpha}$ is a special case of the gerbe $\cL_{\beta}$ constructed in \S \ref{four-term}, (i) is obvious. 

In more detail:  We first observe that \eqref{TEE} provides a quasi-isomorphism between $\cK_2$ (sheaf) and the complex (concentrated in degree zero and one)
\begin{equation}\label{eta}
\eta: \cK_2 \to [K_2^{\eta} \xrightarrow{\nu} Ker(\delta)].
\end{equation}
 
Now, suppose $U$ is disjoint from the support of $\alpha$. On such an open set $U$, the map $\phi$ is zero. This means that the objects $u$ of the category $\cC_{\alpha}(U)$ are elements of ${\rm Ker}(\delta)$. The gerbe $\cC_{\alpha}$, when restricted to $U$, is equivalent to  the Picard stack of $\cK_2$-torsors \cite[Expose XVIII, 1.4.15]{SGA4}: in the complex $[K_2^{\eta} \xrightarrow{\nu} {\rm Ker}(\delta)]$,  one has ${\rm Coker} (\nu) =0$ and ${\rm Ker}(\nu) =\cK_2|_U$. Since for any abelian sheaf $G$, the category $\tors(G)$ is the trivial $G$-gerbe, $\cC_{\alpha}$ is the trivial gerbe with band $\cK_2$ on the complement of the support of $\alpha$. 
 
Now, consider an arbitrary open set $V$ of $X$.  By the exactness of \eqref{TEE}, there is an open covering $\{U_i\}$ of $V$ and sections $u_i \in T(U_i)$ with $t_i = \phi(1)$. Fix $i$ and let $U$ be an open set contained in $U_i$. Then the category $\cC_{\alpha}(U)$ is non-empty. The category $D$ with objects $d\in {\rm Ker}(\delta) \subset T(U)$ and morphisms $\mathrm{Hom}_D(d,d')=$ elements $a \in K_2^{\eta}$ with $\nu(a) = d'-d$. The category $D$ is clearly equivalent to the category of $K_2(U)$-torsors. The functor which sends $d$ to $d +u_i$ is easily seen to be an equivalence of categories between $D$ and $\cC_{\alpha}(U)$.  Thus $\cC_{\alpha}$ is a gerbe with band $\cK_2$. 

(ii) The Bloch-Quillen formula \eqref{BQ} arises from the canonical map
\begin{equation*}
  d^2: Z^2(X) \to H^2(X, \cK_2)
\end{equation*}
of Lemma \ref{comparisons} attached to the four-term complex \eqref{gersten4}. As $\cC_{\alpha}$ is a gerbe of the form $\cL_{\beta}$, (ii) follows from Lemma \ref{comparisons}. 

(iii) This is a simple consequence of the Bloch-Quillen formula (\ref{BQ}).
\end{proof}

\begin{remark}\label{trivial-eta} (i) Split the sequence \eqref{gersten4} into
  \begin{equation*}
    0 \too \cK_2 \too K_2^{\eta} \too K_2^{\eta}/{\cK_2} \too 0
  \end{equation*}
  and
  \begin{equation*}
    0 \too K_2^{\eta}/{\cK_2} \too \bigoplus_{x \in X^{(1)}} j_* K_1(x)  \too \bigoplus_{x \in X^{(2)}} j_* K_0(x) \too 0.
  \end{equation*}

Since the Gersten resolution is by flasque sheaves, one has $H^1(X, K_2^{\eta}/{\cK_2}) \xrightarrow{\sim} H^2(X, \cK_2)$. As Cartier divisors are elements of $H^0(X, K_1^{\eta}/{\cK_1})$, we view elements of $H^1(X, K_2^{\eta}/{\cK_2})$ as \emph{Cartier cycles of codimension two}. The map $Z^1(X) \to H^1(X, K_2^{\eta}/{\cK_2})$ attaches to any cycle its Cartier cycle.  Lemma \ref{diversi} provides the following succinct description of $\cC_{\alpha}$: 

\emph{it is the gerbe of liftings (to a $K_2^{\eta}$-torsor) of the $(K_2^{\eta}/{\cK_2})$-torsor determined by $\alpha$.}
 
 (iii) The proof of Theorem \ref{k2gerbe} provides a canonical trivialization\footnote{This uses \eqref{eta}.} $\eta_{\alpha}$ of the gerbe $\cC_{\alpha}$ on the complement of the support of $\alpha$.
 
 (iv) The pushforward of $\cC_{\alpha}$ along $\cK_2 \to \Omega^2$ produces a $\Omega^2$-gerbe which manifests the cycle class of $\alpha$ in de Rham cohomology $H^2(X, \Omega^2)$. If $\alpha$ is homologically equivalent to zero, then this latter gerbe is trivial, i.e., it is the Picard stack of $\Omega^2$-torsors. \qed\end{remark}

\begin{remark}\label{empty} It may be instructive to compare the $\bbG_m$-torsor $\cO_D$ attached to a divisor $D$ of $X$ and the $\cK_2$-gerbe $\cC_{\alpha}$ attached to a codimension-two cycle. Let $U$ be any open set of $X$. This goes, roughly speaking, as follows. 
\begin{itemize} 
\item $\cO_D$: The set of divisors on $U$ rationally equivalent to zero is exactly the image of $d$ over $U$ in \eqref{gersten1}. So, the set $\cO_D(U)$ is non-empty if $D =0$ in $CH^1(U)$. The sections of $\cO_D$ over $U$ are given by rational functions $f$ on $U$ whose divisor is $D|_U$. In other words, the sections are rational equivalences  between the divisor $D$ and the empty divisor. The set $\cO_D(U)$ is a torsor over $\bbG_m(U)$. 

\item $\cC_{\alpha}$: We observe that the image of $\delta$ in \eqref{gersten4} consists of codimension-two cycles rationally equivalent to zero. So $\cC_{\alpha}$ is non-empty if $\alpha = 0$ in $CH^2(U)$.   Each rational equivalence between $\alpha$ and the empty codimension-two cycle gives an object of $\cC_{\alpha}(U)$. The sheaf of morphisms between two objects is a $\cK_2$-torsor. \qed
\end{itemize}
\end{remark}

The Bloch-Quillen formula \eqref{BQ} states that equivalence classes of $\cK_2$-gerbes are in bijection with codimension-two cycles (modulo rational equivalence) on $X$. We have seen that a codimension-two cycle determines a $\cK_2$-gerbe (an actual gerbe, not just one up to equivalence). It is natural to ask whether the converse holds:  (see Proposition \ref{gerbe2cyclea} in this regard) 
\begin{question}\label{gerbe2cycle}
Does a $\cK_2$-gerbe on $X$ determine an actual codimension-two cycle?
\end{question}

Consider the $\cK_2$-gerbe $\cG_{D, D'}$ attached to a pair of divisors $D, D'$ on $X$. If $\cG_{D,D'}$ determines an actual codimension-two cycle, then any pair $D, D'$ of divisors  determines a canonical codimension-two cycle on $X$. This implies that there is a canonical intersection of Weil divisors and this last statement is known to be false.  So the answer to Question \ref{gerbe2cycle}  is negative in general. 

\subsection{Gerbes and cohomology with support}\label{support}

Let $F$ be an abelian sheaf on a site $\C$. Recall that (see \eg \cite[\S5.1]{Milne})
$H^1(F)$ is the set of isomorphism classes of auto-equivalences of the trivial gerbe $\tors(F)$ with band $F$; more generally, given gerbes $\cG$ and $\cG'$ with band $F$, then the set $\Hom_\C(\cG, \cG')$ (assumed non-empty) of maps of gerbes is a torsor for $H^1(F)$.

Recall also that, for any sheaf $F$ on a scheme $V$, the cohomology $H^*_Z(V, F)$ with support in a  a closed subscheme $Z$ of $V$ fits into an exact sequence \cite[\S 5]{Bloch}
\begin{equation}
\cdots \too H^i_Z(V,F) \too H^i(V,F) \too H^i(V - Z,F) \too H^{i+1}_Z(V, F) \too \cdots ;
\end{equation} 
the exactness of
\begin{equation*}
   H^1(V, F) \too H^1(V- Z, F) \too H^2_Z(V,F) \too H^2(V, F) \too H^2(V -Z, F)
\end{equation*}
leads to an interpretation of the group $H^2_Z(V, F)$: it classifies isomorphism classes of pairs $(\cG, \phi)$ consisting of a gerbe $\cG$ with band $F$ on $V$ and a trivialization $\phi$ of $\cG$ on $V-Z$, i.e., $\phi$ is an equivalence of $\cG|_{V-Z}$ with $\tors(F|_{V-Z})$. 
 
 \subsection{Geometric interpretation of some results of Bloch}  Bloch \cite{Bloch} has proved that: 
\begin{enumerate}
\item \cite[Proposition 5.3]{Bloch} Any codimension-two cycle $\alpha$ on $X$ has a canonical cycle class  $[\alpha] \in H^2_Z(X, \cK_2)$; here $Z$ is the support of $\alpha$. 
\item \cite[Theorem 5.11]{Bloch} If $D$ is a smooth divisor of $X$, then $\mathrm{Pic}(D) = H^1(D, \cK_1)$ is a 
direct summand of $H^2_D(X, \cK_2)$. 
\end{enumerate}

For (1), we note that, by Remark \ref{trivial-eta}, the gerbe $\cC_{\alpha}$ has a trivialization $\eta_{\alpha}$ on $X - Z$.  By the above interpretation of $H^2$ with support, the pair $(\cC_{\alpha}, \eta_{\alpha})$  defines an element of $H^2_Z(X, \cK_2)$; this is the canonical class $[\alpha]$.

For (2), recall that Bloch constructed maps $a: \mathrm{Pic}(D) \to H^2_D(X, \cK_2)$ and $b: H^2_D(X, \cK_2) \to \mathrm{Pic}(D)$ with $b\circ a$ the identity on $\mathrm{Pic}(D)$. We can interpret the map $a$ as follows. Note that any divisor $E$ of $D$ is a codimension-two cycle $\alpha$ on $X$. The $\cK_2$-gerbe $\cC_{\alpha}$ on $X$ has a canonical trivialization $\eta_{\alpha}$ on $X -E$ (and so also on the smaller $X-D$). The association $E \mapsto (\cC_{\alpha}, \eta_{\alpha})$ gives the homomorphism $a: Pic(D) \to H^2_D(X, \cK_2)$.  

These results of Bloch provide a partial answer to Question \ref{gerbe2cycle} summarized in the following 
\begin{proposition}\label{gerbe2cyclea} Let $\cG$ be a $\cK_2$-gerbe on $X$ and let $\beta \in CH^2(X)$ correspond to $\cG$ in the Bloch-Quillen formula \eqref{BQ}.   Let $\phi$ be a trivialization of $\cG$ on the complement $X -D$ of a smooth divisor $D$ of $X$. Then, $\beta$ can be represented by a divisor of $D$ (unique up to rational equivalence on $D$). 
\end{proposition}
Note that the data of $\phi$ is crucial: the map $\mathrm{Pic}(D) \to CH^2(X)$
is not injective in general \cite[(iii), p.~269]{Bloch2}.
\begin{proposition}
  Let $i\colon D\to X$ and $j\colon U=X-D\to X$ be the inclusion maps.  We have the following short exact sequence of Picard 2-stacks
  \begin{equation*}
    \tors (i_*\cK_1^D) \too \ger (\cK_2^X) \too \ger (j_*\cK_2^U).
  \end{equation*}
\end{proposition}
\begin{proof}
  Analyzing the Gersten sequence \eqref{gersten}, \eqref{eq:9} for $\cK_2$ on
  $X$ and $U$, we get the short exact sequence:
  \begin{equation*}
    0\too i_*G_1^D \too  G_2^X \too j_*G_2^U \too 0.
  \end{equation*}
  This gives a short exact sequence of Picard 2-stacks, then
  use~\eqref{eq:8}. Note that $\tors (i_*\cK_1^D)$ is considered as a Picard
  2-stack with no nontrivial 2-morphisms.
\end{proof}
The global long exact cohomology sequence arising from the exact sequence in the
proposition gives  part of the localization sequence for higher Chow groups
\begin{equation*}
  \cdots \too CH^1(D,1) \too CH^2(X, 1) \too CH^2(X-D,1) \too \mathrm{Pic}(D) \too CH^2(X) \too CH^2(X -D) \too 0.
\end{equation*}
This uses the fact that $CH^1(D,0) =\mathrm{Pic}(D)$, that $CH^1(D,1) = H^0(D, \cO^*)$ and $CH^1(D,j)$ is zero for $j >1$ \cite[ (viii), p.~269]{Bloch2}.

\subsection{The two gerbes associated with an intersection of divisors} For a codimension-two cycle of $X$ presented as the intersection of divisors, we know that the $\cK_2$-gerbes in Theorem \ref{k2gerbe} (Gersten gerbe) and  in \S \ref{pairs} (using Theorem \ref{primo}) (Heisenberg gerbe) are equivalent (as their class in $H^2(X, \cK_2)$ corresponds to the class of the codimension-two cycle in $CH^2(X)$ via \eqref{BQ}). We now construct an actual equivalence between them. 
\begin{theorem}\label{comparison} Suppose that the codimension-two cycle $\alpha$ is the intersection $D.D'$ of divisors $D$ and $D'$ on $X$. There is a natural equivalence\footnote{By \S \ref{support}. the set of such equivalences is a torsor over $H^1(X, \cK_2) =CH^2(X,1)$  \cite[\S 2.1]{Stach}.}
   \begin{equation*}
     \Theta: \cC_{\alpha} \to \cG_{D, D'}
   \end{equation*}
   of $\cK_2$-gerbes on $X$.
 \end{theorem}
 \begin{proof}
 By Theorem \ref{primo} and Theorem \ref{k2gerbe}, the classes of the gerbes $\cG_{D,D'}$ and
 $\cC_{\alpha}$  in $H^2(X, \cK_2)$ both correspond to the class of $\alpha$ in
 $CH^2(X)$. This shows that they are equivalent.

 Let us exhibit an actual equivalence. We will construct a functor $\Theta_U:
 \cC_{\alpha}(U) \to \cG_{D, D'}(U)$, compatible with restriction maps $V\subset
 U\subset X$.  
 
 Consider an object $r \in \cC_{\alpha}(U)$. We want to attach to $r$ a $H$-torsor $\Theta_U(r)$ on $U$ in a functorial manner. Each $\Theta_U(r)$ is a $H$-torsor 
which lifts the $\cK_1 \times \cK_1$-torsor $\cO_D \times \cO_{D'}$ on $U$. 
We will describe $\Theta_U(r)$ by means of \v Cech cocycles. Fix an open covering $\{U_i\}$ of $U$ and write $\cC^n(A)$ for \v Cech $n$-cochains with values in the sheaf $A$ with respect to this covering.\medskip

\paragraph{\textit{Step 1.}} Let $a= \{a_{i,j}\}$ and $b = \{b_{i,j}\}$ with $a,b \in \cC^1(O^*)$ be \v Cech 1-cocycles for $\cO_D$ and $\cO_{D'}$. Pick $h = \{h_{i,j}\} \in \cC^1(H)$ of the form
\begin{equation*}
  h_{i,j} = (a_{i,j}, b_{i,j}, c_{i,j}) \in H(U_i\cap U_j). 
\end{equation*}
We need $c_{i,j} \in K_2(U_i \cap U_j)$ such that $h$ is a \v Cech 1-cocycle (for $\Theta_U(r)$, the putative $H$-torsor). 
 Since $a,b$ are \v Cech cocycles, the \v Cech boundary $\partial h$ is of the form
 \begin{equation*}
  \partial h = \{(1,1, y_{i,j,k})\} 
 \end{equation*}
with $y =\{y_{i,j,k}\} \in \cC^2(\cK_2)$ a \v Cech 2-cocycle. This cocycle $y$
represents the gerbe $\cG_{D,D'}$ on $U$. \vspace*{1em}
  
\paragraph{\textit{Step 2.}} Recall that $\cC_{\alpha}$ is the associated stack of the prestack $U \mapsto \cC_{\alpha}(U)$ where the category $\cC_{\alpha}(U)$ has objects $u \in \oplus_{x \in X^1} j_* K_1(x)$ with  $\delta{u} = \phi(1)$ and morphisms from $u$ to $v$ are elements $a \in K_2^{\eta}$ with $\nu(a) = v-u$. 
 Since the category $\cC_{\alpha}(U)$ is non-empty, the class of the gerbe $\cC_{\alpha}$ (restricted to $U$) in $H^2(U, \cK_2)$ is zero. Since $\cC_{\alpha}$ and $\cG_{D, D'}$ are equivalent,  so the class of $\cG_{D, D'}$ in $H^2(U, \cK_2)$ is also zero.\vspace*{1em}
 
\paragraph{\textit{Step 3.}} Consider the case $r$ is given by a pair $(C, g)$ where $C$ is a divisor on $X$ and $g$ is a rational function on $C$. The condition $\delta(r) = \phi(1)$ says $\alpha \cap U$ is the intersection of $U$ with the zero locus of $g$. Assume $g \in \cO_C(C \cap U)$. Given any lifting $\tilde{g} \in \cO_X(U)$ with divisor $C'$ on $U$,  we can write $\alpha \cap U$ as the intersection of the divisors $C \cap U$ and the (principal) divisor $C'$ in $U$. By the results in \S \ref{pairs}, there is a $\cK_2$-gerbe $\cG_{C\cap U, C'}$ on $U$. As $C'$ is principal, its class in $H^1(U, \cK_1)$ is zero; so the class of $\cG_{C\cap U, C'}$ in $H^2(U, \cK_2)$ is zero.\vspace*{1em} 

\paragraph{\textit{Step 4.}} Let  $z =\{z_{i,j,k}\} \in \cC^2(\cK_2)$ be a \v Cech 2-cocycle for $\cG_{C\cap U, C'}$; 

 So $z = \partial w$ is the boundary of a \v Cech cochain $w = \{w_{i,j}\} \in \cC^1(\cK_2)$. Note that 
 $y - z = \partial v$ for a 1-cochain $v$ since $\cG_{C\cap U, C'}$ and $\cG_{D, D'}$  are equivalent as gerbes on $U$: both are trivial on $U$!

 The \v Cech cochain $h' =\{h'_{i,j}\} \in \cC^1(H)$ with
 \begin{equation*}
   h'_{i,j} = (a_{i,j}, b_{i,j}, c_{i,j})(1,1, - w_{i,j})(1,1,- v_{i,j})
 \end{equation*}
 is a \v Cech cocycle and represents the required $H$-torsor $\Theta_U(r)$ on $U$.\vspace*{1em} 

\paragraph{\textit{Step 5.}} The same argument with simple modifications works for a
general object of $\cC_{\alpha}$.
It is easy to check that $\Theta_U$ is a functor, compatible with restriction maps $V\subset
 U\subset X$, and that the induced morphism of gerbes is an equivalence.
\end{proof}

\subsection{Higher gerbes attached to smooth Parshin chains}\label{end} 

By Gersten's conjecture, the localization sequence \cite[\S7 Proposition 3.2]{Quillen}  breaks up into short exact sequences
\begin{equation*}
  0 \too K_i(V) \too K_i(V-Y) \too K_{i-1}(Y) \too 0, \qquad (i> 0)
\end{equation*}
for any smooth variety $V$ over $F$ and a closed smooth subvariety $Y$ of $V$.
Let $j:D \to X $ be a smooth closed subvariety of codimension one of $X$; write $\iota: X -D \to X$ for the open complement of $D$. Any divisor $\alpha$ of $D$ is a codimension-two cycle on $X$; one has a map ${\rm Pic}(D) \to CH^2(X)$ \cite[(iii), p.~269]{Bloch2}. 
This gives the exact sequence (for $i>0$)
\begin{equation*}
   0 \too \cK_i \too \cF_i \too j_* \cK^D_{i-1} \too 0
\end{equation*}
of sheaves on $X$ where $\cF_i = \iota_*\cK^U_i$ is the sheaf associated with the presheaf $U \mapsto K_i(U-D)$. We write $\cK_i^D$ and $\cK_i^U$ for the usual K-theory sheaves on $D$ and $U$ since the notation $\cK_i$ is already reserved for the sheaf on $X$.  
The boundary map
\begin{equation*}
  H^1(D, \cK^D_1) = H^1(X, j_*\cK^D_1) \too H^2(X, \cK_2)
\end{equation*}
is the map $CH^1(D) \to CH^2(X)$. For any divisor $\alpha$ of $D$, the $\cK^D_1$-torsor $\cO_{\alpha}$ determines a unique $j_*\cK^D_1$-torsor $L_{\alpha}$ on $X$. The $\cK_2$-gerbe $\cC_{\alpha}$ (viewing $\alpha$ as a codimension two cycle on $X$) is the lifting gerbe of the $j_*\cK^D_1$-torsor $L_{\alpha}$ (obstructions to lifting to a $\cF_2$-torsor). 

This generalizes to higher codimensions (and pursued in forthcoming work):
\begin{itemize} 
\item  (codimension three) If $\beta$ is a codimension-two cycle of $D$, then the gerbe $\cC_{\beta}$ on $D$ determines a unique gerbe $L_{\beta}$ on $X$ (with band $j_*\cK^D_2$). The obstructions to lifting $L_{\beta}$ to a $\cF_3$-gerbe is a $2$-gerbe $\cG_{\beta}$ with band $\cK_3$ on $X$. This gives an example of a higher gerbe invariant of a codimension three cycle on $X$. Gerbes with band $K_3(F_X)/{\cK_3}$ provide the codimension-three analog of Cartier divisors $H^0(X, K_1(F_X)/{\cK_1})$.

\item (Parshin chains) Recall that a chain of subvarieties
  \begin{equation*}
    X_0 \hookrightarrow X_1 \hookrightarrow X_2  \hookrightarrow X_3 \hookrightarrow \cdots  \hookrightarrow X_n = X
  \end{equation*}
  where each $X_i$ is a divisor of $X_{i+1}$ gives rise to a Parshin chain on $X$. We will call  a Parshin chain smooth if all the subvarieties $X_i$ are smooth. Iterating the previous construction provides a higher gerbe on $X_n=X$ with band $\cK_j$ attached to $X_{n-j}$  (a codimension $j$ cycle of $X_n$). 
 \end{itemize}


\begin{thebibliography}{00}

\bibitem{SGA4} \emph{Th\'eorie des topos et cohomologie \'etale des sch\'emas. {T}ome 2}.
\newblock Lecture Notes in Mathematics, Vol. 270. Springer-Verlag, Berlin-New
  York, 1972.
\newblock S{\'e}minaire de G{\'e}om{\'e}trie Alg{\'e}brique du Bois-Marie
  1963--1964 (SGA 4), Dirig{\'e} par M. Artin, A. Grothendieck et J. L.
  Verdier. Avec la collaboration de N. Bourbaki, P. Deligne et B. Saint-Donat.


\bibitem{Bloch} 
Spencer Bloch.
\newblock {$K{\rm _{2}}$} and algebraic cycles.
\newblock \emph{Ann. of Math. (2)}, 99:349--379, 1974.

\bibitem{Bloch2}
Spencer Bloch.
\newblock Algebraic cycles and higher {$K$}-theory.
\newblock \emph{Adv. in Math.}, 61(3):267--304, 1986.

\bibitem{Breen} Lawrence Breen.
\newblock Extensions du groupe additif.
\newblock \emph{Inst. Hautes \'Etudes Sci. Publ. Math.}, (48):39--125, 1978.

\bibitem{Breen90}
Lawrence Breen.  \newblock Bitorseurs et cohomologie non ab\'elienne.
\newblock In \emph{The {G}rothendieck {F}estschrift, {V}ol.\ {I}}, volume~86 of
  \emph{Progr. Math.}, pages 401--476. Birkh\"auser Boston, Boston, MA, 1990.

\bibitem{Breen94} Lawrence Breen.
\newblock On the classification of {$2$}-gerbes and {$2$}-stacks.
\newblock \emph{Ast\'erisque}, (225):160, 1994.


\bibitem{Breen99} Lawrence Breen.
\newblock Monoidal categories and multiextensions.
\newblock \emph{Compositio Math.}, 117(3):295--335, 1999.

\bibitem{Brylinskib} Jean-Luc Brylinski.
\newblock \emph{Loop spaces, characteristic classes and geometric quantization},
  volume 107 of \emph{Progress in Mathematics}.
\newblock Birkh\"auser Boston, Inc., Boston, MA, 1993.
 
\bibitem{Brylinski} Jean-Luc Brylinski.
\newblock Holomorphic gerbes and the {B}e\u\i linson regulator.
\newblock \emph{Ast\'erisque}, (226):8, 145--174, 1994.
\newblock $K$-theory (Strasbourg, 1992).

\bibitem{chinburg}
  Ph.\ Cassou-Nogu{\`e}s, T.\ Chinburg, B.\ Morin, and M.\ J.\ Taylor.
  \newblock The classifying topos of a group scheme and invariants of symmetric  bundles
  \newblock \emph{J. Proc. Lond. Math. Soc.}, (109):3 1093--1136, 2014.

\bibitem{DeligneB} Jean-Luc Brylinski and Pierre Deligne.
\newblock Central extensions of reductive groups by {$\bold K_2$}.
\newblock \emph{Publ. Math. Inst. Hautes \'Etudes Sci.}, (94):5--85, 2001.

\bibitem{DeligneDG} P.~Deligne, \newblock \emph{La formule de dualit\'e globale,} 
\newblock Expos\'e  XVIII in 
\newblock \emph{Th\'eorie des topos et cohomologie \'etale des sch\'emas. {T}ome 3}.
\newblock Lecture Notes in Mathematics, Vol. 305. Springer-Verlag, Berlin-New
  York, 1973.
\newblock S{\'e}minaire de G{\'e}om{\'e}trie Alg{\'e}brique du Bois-Marie
  1963--1964 (SGA 4), Dirig{\'e} par M. Artin, A. Grothendieck et J. L.
  Verdier. Avec la collaboration de P. Deligne et B. Saint-Donat.


\bibitem{DeligneH3}  Pierre Deligne.
\newblock Th\'eorie de {H}odge. {III}.
\newblock \emph{Inst. Hautes \'Etudes Sci. Publ. Math.}, (44):5--77, 1974.

\bibitem{Deligne0} P.~Deligne, \newblock Le d\'eterminant de la cohomologie.
\newblock In \emph{Current trends in arithmetical algebraic geometry ({A}rcata,
  {C}alif., 1985)}, volume~67 of \emph{Contemp. Math.}, pages 93--177. Amer.
  Math. Soc., Providence, RI, 1987.



\bibitem{Deligne} P.~Deligne, \emph{Le symbole mod\'er\'e},  
Inst. Hautes {\'E}tudes Sci. Publ. Math. No. {\bf 73} (1991), 147--181. 



\bibitem{handbook}Eric~M. Friedlander and Daniel~R. Grayson, editors.
\newblock \emph{Handbook of {$K$}-theory. {V}ol. 1, 2}.
\newblock Springer-Verlag, Berlin, 2005.

\bibitem{Gerstenicm} S.~M. Gersten.
\newblock {$K$}-theory and algebraic cycles.
\newblock In \emph{Proceedings of the {I}nternational {C}ongress of
  {M}athematicians ({V}ancouver, {B}. {C}., 1974), {V}ol. 1}, pages 309--314.
  Canad. Math. Congress, Montreal, Que., 1975.


\bibitem{Giraud} Jean Giraud.
\newblock \emph{Cohomologie non ab\'elienne}.
\newblock Springer-Verlag, Berlin-New York, 1971.
\newblock Die Grundlehren der mathematischen Wissenschaften, Band 179.

\bibitem{GrothendieckSGA7}
A.~Grothendieck, \newblock \emph{Biextensions de faisceaux de groupes.} \newblock In: \emph{Groupes de monodromie en g\'eom\'etrie alg\'ebrique. {I}}.
\newblock Lecture Notes in Mathematics, Vol. 288. Springer-Verlag, Berlin-New
  York, 1972.
\newblock S{\'e}minaire de G{\'e}om{\'e}trie Alg{\'e}brique du Bois-Marie
  1967--1969 (SGA 7 I), Dirig{\'e} par A. Grothendieck. Avec la collaboration
  de M. Raynaud et D. S. Rim.


\bibitem{Hartshorne} Robin Hartshorne.
\newblock \emph{Algebraic geometry}.
\newblock Springer-Verlag, New York-Heidelberg, 1977.
\newblock Graduate Texts in Mathematics, No. 52.


\bibitem{Illusie} Luc Illusie.
\newblock \emph{Complexe cotangent et d\'eformations. {I}}.
\newblock Lecture Notes in Mathematics, Vol. 239. Springer-Verlag, Berlin-New
  York, 1971.


\bibitem{IllusieII} Luc Illusie.
\newblock \emph{Complexe cotangent et d\'eformations. {II}}.
\newblock Lecture Notes in Mathematics, Vol. 283. Springer-Verlag, Berlin-New
  York, 1972.


\bibitem{jardine} J.~F. Jardine.
\newblock Simplicial objects in a {G}rothendieck topos.
\newblock In \emph{Applications of algebraic {$K$}-theory to algebraic geometry
  and number theory, {P}art {I}, {II} ({B}oulder, {C}olo., 1983)}, volume~55 of
  \emph{Contemp. Math.}, pages 193--239. Amer. Math. Soc., Providence, RI, 1986.





\bibitem{L-MB} G{\'e}rard Laumon and Laurent Moret-Bailly.
\newblock \emph{Champs alg\'ebriques}, volume~39 of \emph{Ergebnisse der
  Mathematik und ihrer Grenzgebiete. 3. Folge. A Series of Modern Surveys in
  Mathematics [Results in Mathematics and Related Areas. 3rd Series. A Series
  of Modern Surveys in Mathematics]}.
\newblock Springer-Verlag, Berlin, 2000.



\bibitem{Milne67} J.~Milne.
\newblock The conjectures of {B}irch and {S}winnerton-{D}yer for constant abelian varieties over function fields.
\newblock Thesis, 1967
\newblock  available at \href{http://jmilne.org/math/articles/1967.pdf}{www.jmilne.org}

\bibitem{Milne} J.~Milne.
\newblock {G}erbes and abelian motives.
\newblock 2003.
\newblock Preprint, available at  \href{http://jmilne.org/math/articles/2003b.pdf}{www.jmilne.org}


\bibitem{Parshin} A.~N. Parshin.
\newblock Representations of higher adelic groups and arithmetic.
\newblock In \emph{Proceedings of the {I}nternational {C}ongress of
  {M}athematicians. {V}olume {I}}, pages 362--392. Hindustan Book Agency, New
  Delhi, 2010.


\bibitem{PoonenRains} Bjorn Poonen and Eric Rains.
\newblock Self cup products and the theta characteristic torsor.
\newblock \emph{Math. Res. Lett.}, 18(6):1305--1318, 2011.

\bibitem{Quillen} D.~Quillen,
\newblock Higher Algebraic {K}-theory {I}.
 \newblock Lecture Notes in Mathematics, Vol. 341. 
 \newblock pages 85--137,
 \newblock Springer-Verlag, Berlin-New
  York, 1973.
  
\bibitem{Ramakrishnan}  Dinakar Ramakrishnan.
\newblock A regulator for curves via the {H}eisenberg group.
\newblock \emph{Bull. Amer. Math. Soc. (N.S.)}, 5(2):191--195, 1981.



\bibitem{Stach}
Stefan~J. M{\"u}ller-Stach.
\newblock Constructing indecomposable motivic cohomology classes on algebraic
  surfaces.
\newblock \emph{J. Algebraic Geom.}, 6(3):513--543, 1997.


\bibitem{Tatar}
  A.~Emin Tatar.
\newblock Length 3 complexes of abelian sheaves and {P}icard 2-stacks.
\newblock \emph{Adv. Math.}, 226(1):62--110, 2011.
 \newblock  \url{http://dx.doi.org/10.1016/j.aim.2010.06.012} and
 \newblock \url{http://arxiv.org/abs/0906.2393}


\end{thebibliography}
\end{document}